\documentclass[12pt]{amsart}
\usepackage{amsmath,
amsfonts,
amsthm,
amssymb,
mathrsfs}
\usepackage{amsmath}
\usepackage{amsfonts}
\usepackage{amssymb}

\newcommand{\dom}{\operatorname{dom}}

\newcommand{\comment}[1]{}

\newcommand{\Ht}{\mathrm{ht}}
\newcommand{\CH}{\mathrm{CH}}

\newcommand{\MA}{\mathrm{MA}}
\newcommand{\Seq}[1]{\langle #1 \rangle}

\newcommand{\ZFC}{\mathrm{ZFC}}

\newcommand{\range}{\mathrm{range}}

\newcommand{\rest}{\upharpoonright}
\newcommand{\Acal}{\mathcal{A}}
\newcommand{\Bcal}{\mathcal{B}}

\theoremstyle{plain}
\newtheorem{thm}{Theorem}[section]
\newtheorem{lem}[thm]{Lemma}
\newtheorem{prop}[thm]{Proposition}
\newtheorem{cor}[thm]{Corollary}
\newtheorem{fact}[thm]{Fact}
\newtheorem{claim}[thm]{Claim}
\newtheorem{question}[thm]{Question}

\theoremstyle{definition}
\newtheorem{defn}[thm]{Definition}

\newtheorem{obs}[thm]{Observation}

\begin{document}

\author[H. Lamei Ramandi, S. Todorcevic]{Hossein Lamei Ramandi and Stevo Todorcevic}
\address{Institute f\"{u}r Mathematische Logik und 
Grundlagenforschung  \\
Westf\"{a}lische Wilhelms-Universit\"{a}t 
M\"{u}nster,  Germany}
\email{hlamaira@exchange.wwu.de}

\address{Department of Mathematics \\ University of Toronto,
Toronto \\ Canada}
\email{{\tt hossein@math.toronto.edu}}

\address{Department of Mathematics \\ University of Toronto,
Toronto \\ Canada}
\email{{\tt stevo@math.toronto.edu}}

\address{ Institut de Math\'ematiques de Jussieu, Paris, France}
\email{{\tt stevo.todorcevic@imj-prg.fr}}

\title[can You Take Komjath's Inaccessible Away?]{can You Take Komjath's Inaccessible Away?}

\subjclass[]{}

\begin{abstract} 
In this paper we aim to compare Kurepa trees and Aronszajn trees. 
Moreover, we analyze the effect of large cardinal assumptions on this comparison.
Using the the method of walks on ordinals, we will show 
it is consistent with $\ZFC$ that there is a Kurepa tree and 
every Kurepa tree contains an Aronszajn subtree, if there is an inaccessible cardinal.
This is stronger than Komjath's theorem in \cite{Komjath_Aronszajn_Kurepa},
where he proves the same consistency from two inaccessible cardinals.
Moreover, we prove it is consistent with $\ZFC$ that
there is a Kurepa tree $T$ such that if $U \subset T$
is a Kurepa tree with the inherited order from $T$,
then $U$ has an Aronszajn subtree.
This theorem uses no large cardinal assumption.
Our last theorem immediately implies the following:
If $\MA_{\omega_2}$ holds and $\omega_2$ is not a Mahlo cardinal in $\textsc{L}$
then there is a Kurepa tree with the property that every Kurepa subset has an
Aronszajn subtree.
Our work entails proving a new lemma about Todorcevic's $\rho$ function
which might be useful in other contexts.
\end{abstract}

\maketitle

\section{Introduction}
In this paper we aim to compare Kurepa trees and Aronszajn trees. 
Moreover, we analyze the effect of large cardinal assumptions on this comparison. 
We are interested in the question that to what extent do Kurepa trees contain Aronszajn subtrees.
The first result regarding this question is due to Jensen. 
He showed that there is a Kurepa tree in the constructible universe $\textsc{L}$, 
which has no Aronszajn subtrees.
Todorcevic showed that there is a countably closed forcing which
adds a Kurepa tree with no Aronszajn subtree.
Both Jensen's and Todorcevic's  results are in the negative direction. 
In the positive direction, Komjath proved the following theorem.
\begin{thm}\label{Komjath}\cite{Komjath_Aronszajn_Kurepa}
It is consistent relative to the existence of two inaccessible cardinals that there is a Kurepa tree and 
every Kurepa tree has an Aronszajn subtree.
\end{thm}
It is natural to ask whether or not the large cardinal assumptions in Theorem \ref{Komjath} is sharp. 
In other words, assume every Kurepa tree has an Aronszajn subtree, 
then is it consistent that there are at least two inaccessible cardinals?

Let's call an $\omega_1$-tree \emph{Aronszajn free} if it has no 
Aronszajn subtree.
Without the use of large cardinals, there are various ways  to show the consistency of the existence of Aronszajn free Kurepa trees. 
It is natural to ask, if there are no large cardinals, 
do Kurepa trees have to have Aronszajn free Kurepa subtrees?
In other words, do we need large cardinals in order to show 
the existence of a Kurepa tree with no Aronszajn free Kurepa subtree?

Our work reveals a new fact
about Todorcevic's $\rho$ function. 
Based on this fact about $\rho$ and a notion of capturing which was introduced in \cite{no_real_Aronszajn}, 
we find Aronszajn subtrees in some canonical Kurepa trees
without any large cardinal assumptions.
It is worth mentioning that although we analyze some $\omega_1$-trees to prove this fact about $\rho$,
the function $\rho$  is defined in terms of ordinals with no reference to $\omega_1$-trees.

In this paper we will show the following theorem, which is stronger than Komjath's Theorem.
\begin{thm}\label{main}
Assume there is an inaccessible cardinal. Then it is consistent that there is a Kurepa tree and 
every Kurepa tree contains an Aronszajn subtree.
\end{thm}

Regarding the existence of 
a Kurepa tree with no Aronszajn free Kurepa subtree,
we show the following theorem. It is worth mentioning that the following theorem does not need any large cardinal assumption.

\begin{thm}\label{wlc}
It is consistent that there is a Kurepa tree $T$ such that whenever 
$U \subset T$ is a Kurepa tree when it is considered with the inherited order from $T$, 
then $U$ has an Aronszajn subtree.
\end{thm}
In \cite{walks} by using $\rho$, Todorcevic introduces a forcing which satisfies the Knaster condition 
and which adds a Kurepa tree. 
We use this forcing to prove Lemma \ref{rho},
which reveals a new inequality about $\rho$.
We use Lemma \ref{rho} to find Aronszajn subtrees and show Theorem \ref{wlc}.
Since the tree $T$ can be forced to exist in any model of $\square_{\omega_1}$ using  a ccc forcing,
the following corollary trivially follows from Theorem \ref{wlc}.
\begin{cor}
Assume $\MA_{\omega_2}$ holds and $\omega_2$ is not a Mahlo cardinal in $\textsc{L}$.
Then there is a Kurepa tree with the property that every Kurepa subset has an
Aronszajn subtree.
\end{cor}
The following question still remains unanswered. 
\begin{question}
Is the large cardinal assumption in Theorem \ref{main} sharp? In other words, assume every Kurepa tree has an Aronszajn subtree. Then is it consistent that there is an inaccessible cardinal?
\end{question}

\section{Preliminaries}

We will be using the following notation and terminology. 
Assume $T$ is an $\omega_1$-tree. 
For any $\alpha \in \omega_1$, $T_\alpha$ denotes the set of all elements of $T$ 
which have height $\alpha$.
$T_{< \alpha} $ is the set of all members of $T$ which have height less than $\alpha$.
$T_{\leq \alpha}$ is defined similarly.
$\mathcal{B}(T)$ is the set of all cofinal branches of $T$.
If $b$ is a cofinal 
branch in $T$ and $\alpha \in \omega_1$, $b(\alpha)$ refers to the element in $b$ which is of height $\alpha$.
If $t \in T$ and $\alpha \in \omega_1$ then $t \rest \alpha$ refers to the set of all elements $x \leq_T t$ whose height
is less than $\alpha$. 
For any $x \in T$, $T_x$ is the set of all $t \in T$ that are comparable with $x$.
In particular the predecessors of $x$ are in $T_x$.

If $x$ is a finite set of ordinals and $i \in |x|$ then $x(i)$ refers to the 
$i$'th element of $x.$
For $x,y$ finite sets of ordinals we say $x<y$ if every element in $x$
is less than every element in $y$. 
Assume $x$ is a finite set of ordinals and 
$\Seq{T^\alpha: \alpha \in x}$ are $\omega_1$-trees,
then 
$\bigotimes\limits_{\alpha \in x}T^\alpha = 
\bigcup\limits_{\xi \in \omega_1}
 \prod\limits_{\alpha \in x} T^\alpha_\xi$.
It is easy to see that  
the component-wise order on this product makes it an $\omega_1$-tree. 
With this product, for every $n \in \omega$ we can define $T^{[n]}= \bigotimes\limits_{i \in n } T$.
Assume $T$ is an $\omega_1$-tree and $\Seq{v_i: i \in n}$ are 
pairwise distinct 
elements of $T$ with the same height, then $\bigotimes\limits_{i \in n}T_{v_i}$ is called a derived tree of $T$ with dimension $n$.

In  \cite{no_real_Aronszajn} a notion of capturing is defined for linear orders. 
This notion can be used for $\omega_1$-trees as well. 
We will use this notion and Proposition \ref{Asubtree} in order to characterize 
when an $\omega_1$-tree contains an Aronszajn subtree.

\begin{defn}\cite{no_real_Aronszajn}
Assume $T$ is an $\omega_1$-tree, $\kappa$ is a large enough regular cardinal, $t \in T \cup \mathcal{B}(T)$,
and $N \prec H_\kappa$ is countable such that $T \in N$.
We say that $N$ \emph{captures} $t$ if there is a chain $c \subset T$ in $N$ which contains 
all elements of $T_{< N \cap \omega_1}$ below $t$, or equivalently $t \rest (\delta_N ) \subset c$.
\end{defn}

The following definition is a modification of Definition 3.1 in \cite{no_real_Aronszajn}.
\begin{defn}
Assume $T=(\omega_1,<)$ is an $\omega_1$-tree, $x\in T \cup \mathcal{B}(T)$ and $N \prec H_\theta$ is countable with $T \in N$.
We say that $x$ is \emph{weakly external to} $N$ if there is a stationary $\Sigma \subset [H_{({2^{\omega_1}})^+}]^\omega$
in $N$ such that  
$$\forall M \in N \cap \Sigma \textrm{, } M \textrm{ does not capture }x.$$
\end{defn}
Note that there is a major difference between the definition above and Definition 3.1 in 
\cite{no_real_Aronszajn}. If we require $\Sigma$ to be a club we obtain the definition 
of \emph{external elements} in \cite{no_real_Aronszajn}. 
This is why we call $x$ weakly external in our definition. 
The purpose of this definition is to find Aronszajn suborders. 
It turns out that the existence of weakly external elements is enough for an $\omega_1$-tree
to have Aronszajn subtrees.
This should be compared with Theorem 4.1 in \cite{no_real_Aronszajn}, 
where the existence of external elements is required for finding Aronszajn suborders.
The proof we present here uses the ideas in the proof of
Theorem 4.1 in \cite{no_real_Aronszajn}, but we will include 
it for more clarity.

\begin{prop}\label{Asubtree}
Let $T=(\omega_1,<)$ be an $\omega_1$-tree,
$\kappa = ({2^{\omega_1}})^+$ and 
$\Sigma \subset [H_\kappa]^\omega$ be stationary.
Assume for all large enough regular cardinal $\theta$ there are $x \in T$ and countable $N \prec H_\theta$ such that
$x$ is weakly external to $N$, witnessed by $\Sigma$. In other words, 
for all $M \in \Sigma \cap N$, $M$ does not capture $x$.
Then $T$ has an Aronszajn subtree.
\end{prop}
\begin{proof}
Fix $\theta$ as in the proposition.
For each $t \in T$ let $W_t$ be the set of all countable $N' \prec H_\theta$
such that $\Sigma , T$ are in $N'$ and there is $s>t$ such that for all $M \in \Sigma \cap N'$,
$M$ does not capture $s$.
Let $A$ be the set of all $t \in T$ such that $W_t$ is stationary. We will show that $A$ is Aronszajn.

First note that $A$ is downward closed.
This is because if $t < t'$ then $W_{t'} \subset W_t$.
Moreover, if $t \in T$, $\delta \in \omega_1$ and  $\Ht(t) < \delta$ then $W_t = \bigcup \{ W_s : s>t \textrm{ and } \Ht(s)=\delta  \}$. 
In other words, if $A \neq \emptyset$ then $A$ is uncountable.
So it suffices to show that $A\neq \emptyset$ and $A$ does not contain any uncountable 
branch of $T$.

First we will show that $A \neq \emptyset$. 
Fix a regular cardinal $\lambda > 2^\theta$ such that $\theta$ is definable in $H_\lambda$.
Let $P \prec H_\lambda$ be countable such that for some $x \in T$ 
$\Sigma$ witnesses that $x$ is weakly external to $P$.
Let $t \in T \cap P$ and $t < x$. 
Then $P\cap H_\theta \in W_t \in P$.
By the fact that $P \prec H_\lambda$,  $W_t$ intersects every closed unbounded subset of 
$[H_\theta]^\omega$ which means that it is stationary and $A \neq \emptyset$.

In order to see $A$ contains no uncountable branch of $T$, assume for a contradiction 
that $b \subset A$ is a cofinal branch. 
Let $M \prec H_\kappa$ be countable such that $T,A,b,$ are in $M$ and $M \in \Sigma.$
Let $\delta = M \cap \omega_1$ and $t = b(\delta)$.
Let $N \prec H_\theta$ be countable such that $N \in W_t$ and $M \in N$. 
This is possible because $t \in A$ and $W_t$ is a stationary subset of $[H_\theta]^\omega$.
Let $s>t$ be the element in $T$ such that for all $Z \in \Sigma \cap N$, $Z$ does not capture $s$.
But $M \in \Sigma \cap N$ and it captures $s$ via $b$. This is a contradiction.
\end{proof}
If $T$ is an $\omega_1$-tree
with an Aronszajn subtree $A$, $N \prec H_\theta$ is countable
with $A  \in N$,
and $x \in A \setminus N$, then $x$ is external to $N$.
This makes the following corollary immediate.
\begin{cor}
Assume $T=(\omega_1, < )$ is an $\omega_1$-tree. Then the following are equivalent:
\begin{itemize}
\item $T$ has an Aronszajn subtree.
\item For all large enough regular cardinal $\theta$ there are $x \in T$ and countable 
$N \prec H_\theta$ such that $x$ is external to $N$.
\item For all large enough regular cardinal $\theta$ there are $x \in T$ and countable 
$N \prec H_\theta$ such that $x$ is weakly external to $N$.
\end{itemize}
\end{cor}

We will use the following facts from \cite{Mahlo_gen_ccc} which are due to 
Jensen and Schlechta.
For more clarity we will include the sketch of their proofs.

\begin{fact}\cite{Mahlo_gen_ccc}\label{JS}
Assume $A \in \textsc{V}$ is a countably closed poset, $F \subset A$ is $\textsc{V}$-generic, $B \in \textsc{V}$ is a 
ccc poset and $G \subset B$ is $\textsc{V}[F]$-generic. 
Let $T \in \textsc{V}[G]$ be a normal $\omega_1$-tree. 
\begin{enumerate}
\item If $b \in \textsc{V}[F][G]$ is a cofinal branch in $T$,
then  $b \in \textsc{V}[G]$.
\item If $S \in \textsc{V}[F][G]$ is 
a downward closed Souslin subtree of $T$  then $S \in \textsc{V}[G]$.
\end{enumerate}
\end{fact}
\begin{proof}
Assume for a contradiction that $b \in \textsc{V}[F][G] \setminus \textsc{V}[G]$ is a branch in $T$, 
and let $\dot{b}$ be the name which is forced by $1$ to be outside of $\textsc{V}[G]$.
For $k \in 2,$ let $j_k: A\times B \longrightarrow A^2 \times B$ be the injections which take 
$(p,q)$ to $(1, p,q)$ and $(p, 1, q)$.
Obviously, these injections naturally induce injections on $(A \times B)$-names. 
We will abuse the notation and use $j_k$ for the injections on names too. 
Let $j_k(\dot{b})= \tau_k$ for $k \in 2$.
Since $b \notin \textsc{V}[G]$, $1_{A^2 \times B} \Vdash \tau_0 \neq \tau_1$.

Note that the set  $D=\{(a_0,a_1) \in A^2: \exists \alpha \in \omega_1$
$(a_0, a_1, 1_B) \Vdash \tau_0(\alpha) \neq \tau_1(\alpha) \}$
is dense in $A^2$.
This uses an argument similar to the proof of fullness lemma and the fact that countably closed posets
do not add new countable subsets of the ground model.
Similarly, the set $D_\alpha = \{ a \in A : $ for some $B$-name 
$\dot{x}$, $(a, 1_B) \Vdash \dot{b}(\alpha) = \dot{x} \}$ is dense in $A$.
Now construct an increasing sequence $\alpha_n, n \in \omega$ and $a_s, \dot{x}_s$ for $s \in 2^{< \omega}$
such that:
\begin{itemize}
\item $(a_s, 1) \Vdash \dot{b} (\alpha_{|s|}) = \dot{x}_s$ where $\dot{x}_s$ is a $B$-name in 
$\textsc{V}$,
\item $a_{s\frown 0}, a_{s \frown 1}$ are both below $a_s$ and
 $(a_{s \frown 0}, a_{s \frown 1}, 1) \Vdash  \dot{x}_{s\frown 0} \neq \dot{x}_{s \frown 1}$.
\end{itemize}
For each $r \in 2^\omega \cap \textsc{V}$ let $a_r$ be the lower bound for $\Seq{a_s : s \subset r}$.
In $\textsc{V}[G]$ let $y_r$ be the element which is forced by $a_r$ to be the element on top of $\Seq{x_s : s \subset r}$. This means that $T$ has an uncountable level in $\textsc{V}[G]$ which is a contradiction.

The proof of the statement for Souslin subtrees uses similar ideas and the following facts, which we 
briefly mention. 
First note that if $X$ is a countable subset of $\textsc{V}$
which is in $\textsc{V}[F][G]$ then $X \in \textsc{V}[G]$. 
Also, if $S$ is a Souslin subtree of $T$ in $\textsc{V}[G]$ 
then it is Souslin in $\textsc{V}[F][G]$. 
If $S \in \textsc{V}[F][G] \setminus \textsc{V}[G]$ is a downward closed Souslin subtree of $T$ 
then there is downward closed Souslin $S' \subset S$ such that every cone $S'_x$ is outside of $\textsc{V}[G]$ for all $x \in S'$.

Now assume for a contradiction that $S$ is a Souslin subtree of $T$ which is in 
$\textsc{V}[F][G] \setminus \textsc{V}[G].$ 
Without loss of generality we can assume that every cone $S_x$ is outside of $\textsc{V}[G]$, 
for every $x \in S$. 
Assume $\dot{S}$ is the name which is forced by $1$ to be outside of $\textsc{V}[G]$.
Again let $\tau_k$ be the corresponding names $j_k(\dot{S})$ as above.

Let $S_k$ be the Souslin tree for $\tau_k$, for $k \in 2,$ in the extension by 
$(F_0,F_1,G) \subset A^2 \times B$ which  is $\textsc{V}$-generic.
Note that $S_0 \cap S_1 \subset T_{< \alpha}$ for some $\alpha \in \omega_1$. 
In order to see this assume $S_0 \cap S_1$ is uncountable. 
Then $S_0 \cap S_1$ is an uncountable downwards closed subtree of $S_0 \cup S_1$.
But $S_0 \cup S_1$ is a Souslin tree.
So $S_0 \cap S_1$ contains a cone from $S_0 \cup S_1$.
Then for some $x \in S_0 \cap S_1$, $(S_0)_x = (S_1)_x$.
But then $(S_0)_x = (S_1)_x \in \textsc{V}[F_0][G] \cap \textsc{V}[F_1][G]= \textsc{V}[G]$, 
which is a contradiction.
Choose an increasing sequence $\Seq{\alpha_n: n \in \omega }$ and a sequence 
$\Seq{a_s, \dot{x}_s : s  \in 2^{<\omega}}$ such that:
\begin{itemize}
\item $(a_s, 1) \Vdash \dot{S} \cap T_{< \alpha_{|s|}} = \dot{x}_s$, where $\dot{x}_s$ is a $B$-name in $\textsc{V}$,
\item $a_{s\frown 0}, a_{s \frown 1}$ are both below $a_s$,
\item 
$(a_{s \frown 0}, a_{s\frown 1}, 1) \Vdash T_{\alpha_s} \cap \dot{x}_{s \frown 0} \cap \dot{x}_{s\frown 1} = \emptyset$.
\end{itemize}
For each $r \in 2^\omega \cap \textsc{V}$, let $a_r$ be a lower bound for $\Seq{a_s: s \subset r}$.
Also let $\alpha = \sup \{ \alpha_n : n \in \omega\}$.
Now we work in $\textsc{V}[G]$.
For each $r$ let $x_r = \bigcup\limits_{s \subset r}\dot{x}_s[G]$.
Note that if $r \neq r'$ then there is no $t \in T_\alpha$ such that the set of predecessors of $t$ is contained in $x_r \cap x_{r'}$.
For each $r$, let $y_r \in T_\alpha$ such that $\{ t \in T: t < y_r \} \subset x_r$.
But this means that $T_\alpha$ is uncountable which is a contradiction.
\end{proof}

In this paper $ \textrm{coll}(\omega_1, < \lambda)$
refers to the usual Levy collapse forcing with countable 
conditions which collapses every cardinal less than $\lambda$
to $\omega_1$.
Fact \ref{JS} immediately implies the following lemma. 

\begin{lem}\label{small_extension}
Let $\lambda \in \textsc{V}$ be an inaccessible cardinal,
$F \subset \textrm{coll}(\omega_1, < \lambda)$ be 
$\textsc{V}$-generic, $\mathbb{P}$ be a ccc poset of size $\aleph_1$ in 
$\textsc{V}[F]$, $G \subset \mathbb{P}$ be $\textsc{V}[F]$-generic
and $U \in \textsc{V}[F][G]$ be an $\omega_1$-tree. 
Then $U$ has at most $\aleph_1$ many Souslin subtrees and 
cofinal branches in $\textsc{V}^\mathbb{P}$.
\end{lem}
\begin{proof}
For every $\alpha \in \lambda$, let 
$F_\alpha = F \cap \textrm{coll}(\omega_1, <\alpha)$. 
Let $\kappa< \lambda$ be a regular uncountable cardinal
such that $\mathbb{P} \in \textsc{V}[F_\kappa]$ and 
$U \in \textsc{V}[F_\kappa][G].$ 
Fact \ref{JS} implies the following. 
\begin{itemize}
\item If $b \in \textsc{V}[F][G]$ is a cofinal branch of $U$ then it is in 
$\textsc{V}[F_\kappa][G]$.
\item If $S \in \textsc{V}[F][G]$ is a Souslin subtree of $U$ then it is in 
$\textsc{V}[F_\kappa][G]$.
\end{itemize}
It is obvious that $|\mathcal{B}(T) \cap \textsc{V}[F_\kappa][G]|= \aleph_1$, in $\textsc{V}[F][G]$.
Similarly the conclusion follows for Souslin subtrees of $U$.
\end{proof}

The following Lemma from \cite{club_isomorphic} is useful in finding 
club embeddings between $\omega_1$-trees.
\begin{lem}[Lemma 3.2 of \cite{club_isomorphic}]\label{AS}
Assume $R$ and all its derived trees are Souslin, 
$A$ is an Aronszajn tree and $R'$ is a derived tree of $R$ whose 
dimension is $n$.
Moreover assume forcing with $R'$ adds a new branch to $A$
and $R'$ has the least dimension with respect to this property 
among the derived trees of $R$.
Then  $R'$ club embeds into $A$.
\end{lem}

\begin{lem}\label{AddAntichain}
For every Aronszajn tree $A$ there is a forcing $P_A$ which
\begin{itemize}
\item adds an uncountable antichain to $A$, 
\item preserves cardinals and  
\item adds no new cofinal branch to  $\omega_1$-trees of the ground model.
\end{itemize}  
\end{lem}
\begin{proof}
For every $\omega_1$-tree $A$, let $P_A$ be the poset consisting 
of all finite antichains in $A$.
Based on the work in \cite{embedding_Atrees_rationals}, 
if $A$ is Aronszajn and $W$ is an uncountable collection of pairwise disjoint finite antichains of $A$, then there are distinct $x,y$ in $W$ such that 
$x \cup y \in P_A$.
Moreover, $P_A$ is ccc if and only if $A$ is Aronszajn.
Since $P_A$ is a ccc poset of size $\aleph_1$, it preserves 
cardinals.
Usual density arguments, shows that $P_A$ adds an uncountable antichain to $A$.

First we show that $P_A$ does not add branches to the 
Aronszajn trees of the ground model. Assume $U$ is an 
Aronszajn tree. 
Without loss of generality assume $U,A$ are disjoint. 
Obviously $A\cup U$ with $<_A \cup <_U$ is an Aronszajn tree. 
Define $\varphi$ from  $P_A \times P_U$ to $P_{A\cup U}$ by
$\varphi(a,b) = a\cup b$.
Observe that $\varphi$ is an isomorphism.
Therefore $P_A \times P_U$ is ccc.
Hence $U$ remains Aronszajn after forcing with $P_A$.

Now we show that $P_A$ does not add new cofinal branches to 
$\omega_1$-trees of the ground model. 
To this end, let $U$ be an $\omega_1$-tree in the ground model, 
to which $P_A$ adds a new branch. Since 
$P_A$ is ccc, the possible points of the new branch 
forms a Souslin subtree of $U$. In particular, 
there is a Souslin tree in the ground model to which 
$P_A$ adds a cofinal branch.
But this is impossible because of what we just showed above.
\end{proof}

It is worth pointing out that in the presence of $\CH$
there are posets which 
in addition to satisfying the requirements of 
Lemma \ref{AddAntichain}, 
do not add new reals.
This is the poset introduced in Remark 5.2 of \cite{first}.
Let $Q_S$ be the poset consisting of all $q=(X_q, \mathcal{U}_q)$
such that:
\begin{itemize}
\item $X_q$ is a countable downward closed subset of $S$ which has a 
last level of height $\alpha_q$, 
\item $\mathcal{U}_q$ is a non-empty countable set and for every 
 $U \in \mathcal{U}_q$  there exists $n \in \omega$ such that $U$ is a 
pruned  downwards closed subtree of $S^{[n]}$,
\item for every $U \in \mathcal{U}_q$ there is a $\sigma \in U$ 
which is a subset of the last level of $X_q$.
\end{itemize}
We let $p \leq q$ if $(X_p)_{\leq \alpha_q} = X_q$ and $\mathcal{U}_q \subset \mathcal{U}_p$.

Observe that for every $q \in Q_S$ and $s \in S$
there are $t>s$ and $p<q$ such that $\alpha_p > \Ht(t)$ and 
$t \notin X_q$.
This shows that if $G \subset Q_S$ is generic then 
$\bigcup\limits_{p \in G} X_p$ does not contain any cone $S_s$.
Obviously, $\bigcup\limits_{p \in G} X_p$ is uncountable downward closed.
Therefore, the minimal elements of $S \setminus \bigcup\limits_{p \in G} X_p$ forms an uncountable antichain in $S$.

Lemma 5.3 of \cite{first} asserts that there exists 
a poset which projects onto $Q_S$ 
and which
does not add 
new branches to $\omega_1$-trees of the ground model.
Therefore, $Q_S$ does not add new branches to $\omega_1$-trees of the ground model.
The fact that $Q_S$ preserves cardinals follows from 
Remark 5.2 in \cite{first}. $\CH$ is only used for preserving 
$\omega_2$.
The same remark also explains why $Q_S$ does not add new reals. 

We will use $\square_{\omega_1}$ in order to have the structure of walks on ordinals up to 
$\omega_2$. The following is the standard definition of $\square_{\omega_1}$.
\begin{defn}
A sequence $\Seq{C_\alpha : \alpha \textrm{ is limit and } \omega_1 < \alpha < \omega_2}$ is said to be a \emph{$\square_{\omega_1}$-sequence} if 
\begin{itemize}
\item $C_\alpha$ is a closed unbounded subset of $\alpha$,
\item $\textrm{otp}(C_\alpha) < \alpha$ and
\item if $\alpha$ is a limit point of $C_\beta$ then $C_\beta \cap \alpha= C_\alpha $.
\end{itemize}
The assertion that there is a $\square_{\omega_1}$-sequence is  called $\square_{\omega_1}.$
\end{defn}
The following proposition is obtained from standard argument using $\square_{\omega_1}$-sequences.
\begin{prop} \label{square}
If  $\square_{\omega_1}$ holds then there is a sequence 
$\Seq{C_\alpha : \alpha \in \omega_2}$ such that 
\begin{itemize}
\item $C_\alpha$ is a closed unbounded subset of $\alpha$,
\item $C_{\alpha +1} = \{ \alpha \}$,
\item $\textrm{otp}(C_\alpha) \leq \omega_1$ 
and if 
 $\textrm{cf}(\alpha) = \omega$ then $\textrm{otp}(C_\alpha)< \omega_1$,
\item if $\alpha \in C_\beta$ and $\beta$ is limit then $\textrm{cf}(\alpha) \leq \omega$,
\item if $\alpha$ is a limit point of $C_\beta$ then $C_\beta \cap \alpha= C_\alpha $.
\end{itemize}
 
\end{prop}
We only consider $\square_{\omega_1}$-sequences which have the properties 
mentioned in the proposition above. We will also use the following standard fact.

\begin{fact}
Assume $\lambda$ is a regular cardinal  which is not Mahlo in $\textsc{L}$. 
Let $G \subset \textrm{coll}(\omega_1 , < \lambda)$ be $\textsc{L}$-generic.
Then $\square_{\omega_1}$ holds in $\textsc{L}[G]$.
\end{fact}

Now we briefly review some definitions and facts about walks on ordinals, from sections 
7.3, 7.4, and 7.5 of \cite{walks} unless otherwise is mentioned.
We fix a $\square_{\omega_1}$-sequence 
$\Seq{C_\alpha: \alpha \in \omega_2}$ which satisfies the properties in Proposition \ref{square}.

We will use the following notation in the rest of the paper.
For all $X$, $\alpha_X = \sup (X \cap \omega_2)$.
For each $\alpha \in \omega_2$ we let $L_\alpha$ be the set of all $\beta \in \omega_2$ 
such that $\alpha \in \lim (C_\beta)$.
For each $\alpha < \beta $ in $\omega_2$, 
let $\Lambda (\alpha, \beta)$ be the maximal limit point of 
$C_\beta \cap (\alpha + 1 )$ when such a limit point exists, otherwise 
$\Lambda(\alpha, \beta) = 0$.
\begin{defn}[See section 7.3 in \cite{walks}]
The function $\rho : [\omega_2]^2 \longrightarrow \omega_1$ is defined recursively as follows: for 
$\alpha < \beta$, 
\begin{center}
$\rho(\alpha, \beta ) = \max \{ \textrm{otp}(C_\beta \cap \alpha), 
\rho(\alpha, \min (C_\beta \setminus \alpha)), 
\rho(\xi , \alpha)  : \xi \in C_\beta \cap [\Lambda (\alpha, \beta), \alpha ) \}$.
\end{center}
We define $\rho(\alpha, \alpha)= 0$ for all $\alpha \in \omega_2$. 
When the order between $\alpha, \beta$ is not known we use $\rho\{ \alpha, \beta \}$ instead of $\rho(\alpha, \beta)$. More precisely, $\rho \{\alpha , \beta \}=\rho (\alpha, \beta)$ if $\alpha \leq \beta$ and 
$\rho \{\alpha , \beta \}=\rho (\beta, \alpha)$ if $\beta \leq \alpha$.
\end{defn}
\begin{lem}[Lemma 7.3.6 of \cite{walks}]
Assume $\xi \in \alpha$ and $\alpha$ is a limit point of $C_\beta$. 
Then $\rho(\xi,\alpha) = \rho(\xi, \beta).$
\end{lem}
\begin{lem}[Lemma 7.3.11 of \cite{walks}] \label{cofinal_sequence}
If $\alpha < \beta$, $\alpha $ is a limit ordinal such that there is a cofinal sequence of $\xi \in \alpha$,
with $\rho(\xi, \beta) \leq \nu$
then $\rho(\alpha, \beta) \leq \nu$.
\end{lem}

\begin{lem}[ Lemma 7.3.8 of \cite{walks}] \label{countable_preimage}
For all $\nu \in \omega_1$ and $\alpha \in \omega_2$, the set 
$\{ \xi \in \alpha : \rho(\xi, \alpha ) \leq \nu \}$ is countable.
\end{lem}

\begin{lem}[Lemma 7.3.7 of \cite{walks}]
Assume $\alpha \leq \beta \leq \gamma$. Then 
\begin{itemize}
\item $\rho(\alpha, \gamma) \leq \max \{ \rho(\alpha , \beta), \rho(\beta, \gamma) \}$, 
\item $\rho(\alpha, \beta) \leq \max \{ \rho(\alpha, \gamma), \rho(\beta , \gamma) \}$.
\end{itemize}
\end{lem}

The following lemma can be obtained in the same way as Lemma 3.1.3 of \cite{walks}.

\begin{lem}\cite{walks}\label{bluered}
Assume $\alpha < \beta < \gamma$. We have $\rho(\alpha, \gamma) = \rho(\alpha, \beta)$, 
if $\rho(\beta , \gamma) < \max \{\rho(\alpha, \beta),\rho(\alpha, \gamma)  \}$.
\end{lem}
\begin{proof}
We only prove that if $\rho(\alpha, \gamma) > \rho(\beta, \gamma)$ then 
$\rho(\alpha, \gamma) = \rho(\alpha, \beta)$.
The other half of the statement can be proved by similar argument.
$\rho(\alpha, \beta) \leq \max \{ \rho (\alpha , \gamma), \rho (\beta, \gamma) \} = \rho(\alpha , \gamma)$. On the other hand, 
$\rho(\alpha, \gamma) \leq \max \{ \rho(\alpha, \beta), \rho(\beta, \gamma)\}$
and $\rho(\alpha, \gamma) > \rho(\beta, \gamma)$.
So $\rho(\alpha, \gamma) \leq \rho (\alpha, \beta)$.
And this finishes the proof.
\end{proof}

\begin{lem}[Lemma 7.3.10 of \cite{walks}] \label{limit_in_trace}
Assume $\beta \in \lim(\omega_2)$, and $\gamma > \beta$. Then there is $\beta' \in \beta$ 
such that for all $\alpha \in (\beta', \beta)$, $\rho(\alpha, \gamma) \geq \rho(\alpha, \beta)$.
\end{lem}

\begin{lem}[Lemma 7.4.7 of \cite{walks}]\label{cc}
Assume $A$ is an uncountable family of finite subsets of $\omega_2$ and $\nu \in \omega_1$. 
Then there is an uncountable $B \subset A$ such that $B$ forms a $\Delta$-system with root 
$r$ and for all $a, b$ in $B$:
\begin{itemize}
\item $\rho \{ \alpha , \beta \} > \nu$ for all $\alpha \in a \setminus b$ and $\beta \in  b \setminus a$,
\item $\rho \{ \alpha , \beta \} \geq \min \{ \rho (\alpha, \gamma), \rho(\beta , \gamma)\}$
for all $\alpha \in a \setminus b $, $\beta \in b \setminus a $ and $\gamma \in a \cap b$.
\end{itemize}
\end{lem}

The following forcing is from the proof of Theorem 7.5.9 in \cite{walks}.
\begin{defn}\cite{walks}\label{Q}
Assume $A \subset \omega_2$.
$Q_A$ is the poset consisting of all finite functions $p$ such that the following holds.
\begin{enumerate}
\item $\dom(p) \subset A$. 
\item \label{level condition}
For all $\alpha \in \dom(p)$, $p(\alpha) \in [\omega_1]^{< \omega}$ such that 
for all $\nu \in \omega_1$, $p(\alpha) \cap [\nu, \nu + \omega )$ has at most one element.
\item \label{initial segment condition}
For all $\alpha, \beta$ in $\dom(p)$, $p(\alpha) \cap p(\beta)$ is 
an initial segment of both $p(\alpha)$ and $p(\beta)$.
\item \label{delta condition}
For all $\alpha < \beta$ in $\dom(p)$, $\max (p(\alpha) \cap p(\beta)) < \rho(\alpha , \beta)$ or $p(\alpha) \cap p(\beta) = \emptyset$.
\end{enumerate}
We let $q \leq p$ if $\dom(p) \subset \dom(q)$ and $\forall \alpha \in \dom(p)$, $p(\alpha) \subset q(\alpha)$.
We use $Q$ in order to refer to $Q_{\omega_2}.$
The poset $Q_c$ consists of all conditions $p$ in $Q$ with the additional condition that 
for all $\alpha \in \dom(p)$, $\textrm{cf}(\alpha) \leq \omega$. 
\end{defn}

\begin{defn}
Assume $G$ is generic for $Q$. Then for each $\xi \in \omega_2,$ $b_\xi =\bigcup \{ p(\xi) : p \in G \}$.
\end{defn}

Recall that a poset $P$ satisfies the \emph{Knaster condition}  if every uncountable subset $A$ of $P$
contains an uncountable subset $B$ such that the elements of $B$ are pairwise compatible.
Note that Knaster condition is stronger than ccc.
Moreover,  if $P$ satisfies the Knaster condition then it does not add new cofinal branches to $\omega_1$-trees 
and its iteration with any  ccc poset is ccc. 
\begin{prop}[Theorem 7.5.9. in \cite{walks}] 
The poset $Q$ satisfies the Knaster condition.
\end{prop}

We finish this section by some simple observations regarding the poset $Q$.
Let $G \subset Q$ be generic. 
For $t \in s \in \omega_1$, we let $t < s$ if there is $\alpha \in \omega_2$ 
and $p \in G$ such that $\alpha \in \dom(p)$ and $t, s$ are in $p(\alpha)$.
By Condition \ref{initial segment condition} of Definition \ref{Q}, $<$ is transitive 
and $T = (\omega_1, <)$ forms a tree.
Also note that for all $\alpha \in \omega_2$ and $\nu \in \{ 0 \} \cup \lim(\omega_1)$
the set of all conditions
 $q \in Q$ such that $ \alpha \in \dom(q) \textrm{ and } q(\alpha) \cap [ \nu , \nu + \omega) \neq \emptyset$ is a dense subset of $Q$.
So for each $\alpha \in \omega_2$ and $ \nu \in \{0 \} \cup \lim(\omega_1)$, 
$|b_\alpha \cap [ \nu , \nu + \omega)| = 1$.
This means that for each $\alpha \in \omega_2$, $b_\alpha$ is a maximal uncountable 
branch of $T$.
Similar arguments show that if $s \neq t$ have limit heights in $T$ 
then they have different sets of predecessors.
In particular $T$ is normal.

Moreover, it is easy to see that for all $t \in T$, the set of all $q \in Q$ such that 
for some $\alpha \in \dom (q)$, $ t \in q(\alpha)$ forms a dense subset 
of $Q$. 
This means that for each $t \in T$ there is $\alpha \in \omega_2$ such that
$t \in b_\alpha$.
Therefore, for each $\nu \in \{ 0 \} \cup \lim (\omega_1)$, the set 
$[\nu, \nu + \omega)$ is a level of the tree $T$.
In particular $T$ is an  $\omega_1$-tree whose levels are countable sets that are in the ground model.

\section{Complete Suborders of Q}

When we analyze subtrees of the generic tree $T$, which is added by $Q$, 
it will be useful to know if there is a complete suborder of $Q$ which adds the tree $T$ but does not add certain branches.
In this section we will find some subsets of $Q$ which are complete suborders of it.

\begin{lem}
The poset $Q_c$ is a complete suborder of $Q$. 
Moreover, if $X \subset \omega_2$ is a set of ordinals of cofinality 
$\omega_1$, then $Q_{\omega_2 \setminus X}$ 
is a complete suborder of $Q$. 
\end{lem}
\begin{proof}
We only prove the first part of the lemma. The second part can be verified by a similar argument.
Assume $q \in Q$. We will show that there is $q' \in Q_c$ such that for all extensions $p \leq q'$ in $Q_c$,
the conditions $p,q$ are compatible.
Without loss of generality we can assume that $q$ has the following extra property: For all  $\xi < \eta$ in $\dom (q)$
there are distinct $m, n$ in $\omega$ such that $\max(q(\xi) \cap q(\eta)) + \omega + m \in q(\xi)$ and 
$\max(q(\xi) \cap q(\eta)) + \omega + n \in q(\eta)$.
In particular, $q$ decides $\max(b_\xi \cap b_\eta)$ in the generic tree.

Assume $\{ \beta_i : i \in n \}$ is the increasing enumeration
of all ordinals in $\dom(q)$ which have cofinality $\omega_1$. 
Also let $C$ be the set of all ordinals in $\dom(q)$ which have countable cofinality.
Define $\beta'_i$, for each $i \in n$, to be the least ordinal $\xi$ such that:
\begin{enumerate}
\item $\xi$ is a limit point of $C_{\beta_i}$, 
\item $\xi$  is strictly above all elements of $\dom(q) \cap \beta_i$,
\item for all $\alpha \in \dom(q) \setminus \beta_i$, $\rho(\xi, \alpha) \geq \rho(\beta_i, \alpha)$,
\item $\textrm{otp}(C_\xi) > \max(q(\beta_i))$
\end{enumerate}
Note that the third requirement can easily be arranged by Lemma \ref{countable_preimage}. 
Let  $q'$ be the condition in $Q_c$ such that $\dom(q') = C \cup \{ \beta'_i: i \in n \}$, 
$q'(\alpha) = q(\alpha) $ for all $\alpha \in C$, and $q'(\beta'_i) = q(\beta_i)$ for each $i \in n$.
It is easy to see that $q' \in Q_c$. 

Now let $p< q'$ be in $Q_c$. 
Let $r$ be the condition in $Q$ such that:
\begin{enumerate}
\item $\dom(r) = \dom(p) \cup \{ \beta_i: i \in n \}$, 
\item $r(\alpha) = p(\alpha)$ for each $\alpha \in \dom(p)$, and
\item $r(\beta_i) = p(\beta'_i) \cap (\max(q(\beta_i))+1)$.
\end{enumerate} 
 
It is easy to see that $r$ is a common extension of $p $ and $q$, 
provided that it is in $Q$.
We only show that  Condition \ref{delta condition}, of Definition \ref{Q} holds for $r$.
Assume $\alpha \in \dom(p)$ and $\beta $ is one of the $\beta_i$'s. 
If $\alpha < \beta'$ then $\rho(\alpha, \beta') = \rho(\alpha, \beta)$. Therefore,
$\max (r(\alpha) \cap r(\beta)) \leq \max (p(\alpha) \cap p(\beta'))< \rho(\alpha, \beta') = \rho (\alpha, \beta)$, 
which was desired.
If $\beta' \leq \alpha < \beta $, then  
$$\rho(\alpha, \beta) \geq \textrm{otp}(C_{\beta'}) > \max(q(\beta)) = \max (r(\beta)) \geq \max (r(\beta) \cap r(\alpha)).$$
If $\beta< \alpha$ note that by lemma \ref{bluered} 
either $\rho(\beta, \alpha) \geq \rho(\beta', \alpha)$ or
$\rho(\beta', \beta) = \rho(\beta', \alpha)$.
In the first case there is nothing to show, and for the second case 
we have $\rho(\beta', \alpha) =\rho(\beta', \beta) = \textrm{otp}(C_{\beta'}) > \max(q(\beta))$.

Now assume that $\alpha < \beta$ are both in $\{\beta_i: i \in n \}$.
Then  $$\max (r(\alpha) \cap r(\beta)) \leq \max (p (\alpha') \cap p(\beta'))=\max (q'(\alpha') \cap q'(\beta')).$$
Here the inequality is obvious. The equality follows from the facts that $q'(\beta_i') = q(\beta_i)$, for each $i \in n$,
and $q$ satisfies the extra property in the beginning of the proof. Moreover,
$$\max (q'(\alpha') \cap q'(\beta'))=\max (q(\alpha) \cap q(\beta)) < \rho (\alpha, \beta).$$
This assures us that $r$ satisfies condition \ref{delta condition} of Definition \ref{Q}.
\end{proof}
It is well known that if there is a ccc poset $P$ which adds a branch $b$ to an $\omega_1$-tree $U$,
then $\{ u \in U: \exists p \in P, p\Vdash u \in \dot{b} \}$ is a Souslin subtree of $U$.
Here, $Q$ is a ccc poset and $Q_c$ is a complete suborder of $Q$. Moreover,
if $G \subset Q$ is generic then $G \cap Q_c$ knows the generic  tree $T$.
Since there is a ccc poset $R$ such that $Q$ is equivalent to $Q_c * \dot{R}$,
$T$ has lots of Souslin subtrees in any extension by $Q_c$.
This leads to the following corollary.
In the next section we prove a stronger statement which we will use to 
prove a fact about $\rho$. 
For now, this corollary helps us to have a better picture of the forcing $Q$.
\begin{cor}
The generic tree for $Q_c$ has Souslin subtrees.
\end{cor}

\begin{lem} \label{projection}
Assume $\CH$. Let $\Seq{N_\xi : \xi \in \omega_1}$ 
be a continuous $\in$-chain of countable elementary 
submodels of $H_\theta$ where $\theta$ is a regular large enough cardinal, 
$N_{\omega_1} = \bigcup_{\xi \in \omega_1} N_\xi$, and
$\mu = \sup (N_{\omega_1} \cap \omega_2)$.
Then $Q_\mu$ is a complete suborder of $Q$.
\end{lem}
\begin{proof}
We need to show that for all $q \in Q$ there is $p \in Q_\mu$ such that 
if $r \leq p$ and $r \in Q_\mu$ then $r$ is compatible with $q$.
Let $R = \bigcup \range(q)$, $L= \dom(q) \cap \mu$, and $H = \dom(q) \setminus \mu = \{\beta_i : i \in k \}$  
such that $\beta_i$ is increasing.
Fix $\bar{\nu} \in \omega_1$ which is above all elements of $R$ and all $\rho(\alpha , \beta)$
where $\alpha , \beta$ are in $\dom(q)$.
Using Lemma \ref{countable_preimage}, fix $\mu_0 \in \mu$ above $\max (L)$ such that for all 
$\beta \in H$ and for all $\gamma \in \mu \setminus \mu_0$,
$\rho (\gamma, \beta ) > \bar{\nu}$.
For each $\beta \in H$ and $\nu \in \bar{\nu}$ let 
$A_{\nu , \beta} = \{ \alpha \in \mu_0 : \rho (\alpha , \beta ) = \nu \}.$

Again by Lemma \ref{countable_preimage}, for all $\nu \in \bar{\nu}$ and $\beta \in H$, 
$A_{\nu , \beta}$ is a countable subset of $\mu_0$.
Since  $\CH$ holds, we can fix $N = N_\xi$ such that 
$\mu_0, \bar{\nu}, L, R, \Seq{A_{\nu, \beta} : \beta \in H, \nu \in \bar{\nu}}$ are in $N$.
By elementarity, there is $H' = \{ \beta_i' : i \in k \}$ which is in $N$ and 
\begin{enumerate}
\item $\beta_i'$ is increasing,
\item $\min(H') > \mu_0$
\item for all $i \in k$ and for all $\nu \in \bar{\nu}$, 
$A_{\nu , \beta_i}  = \{ \alpha \in \mu_0 : \rho(\alpha, \beta_i') = \nu  \}$, and 
\item for all $i<j$ in $k$, $\rho (\beta_i, \beta_j) = \rho (\beta_i' , \beta_j')$.
\end{enumerate}

Let $p$ be the condition such that $\dom(p) = L \cup H'$, for all $\xi \in L$, $p(\xi) = q(\xi)$ and for all $i \in k$,
$p(\beta_i') = q(\beta_i)$. Suppose $r \leq p$ is in $Q_\mu$. 
We will find $s \in Q$ which is a common 
extension of $r,q$. Pick $s$ such that
$\dom(s) = \dom(r) \cup H$,
$s \rest \dom(r) = r$, and for all $i \in k$ $s(\beta_i) = r(\beta_i') \cap (\max(q(\beta_i)) + 1)$.

We need to show that $s$ is a  condition in $Q$. All of the  conditions in Definition \ref{Q} obviously hold, 
except for condition \ref{delta condition}.
If $\alpha < \beta$ are in $H$, 
by the last requirement for $H'$ and the fact that $r$ is a condition, 
$\max(s(\alpha) \cap s(\beta)) < \rho(\alpha, \beta)$.

Now assume that $\alpha \in \dom(r)$ and $\beta = \beta_i \in H$.
If $\rho (\alpha , \beta) \geq \bar{\nu}$, everything is obvious because $\max (s(\beta)) < \bar{\nu}$.
Assume $\rho(\alpha, \beta) = \nu < \bar{\nu}$.
So $\alpha \in A_{\nu , \beta}$. 
Since $ r \in Q_\mu$, we have 
$\max(s(\alpha) \cap s(\beta_i)) \leq \max (r(\alpha) \cap r(\beta_i')) < \nu$. 
\end{proof}

Lemma \ref{projection} shows
that for many ordinals $\mu$ with cofinality $\omega_1$, $Q_\mu$ is a complete suborder of $Q$. 
It is natural to ask the same question for ordinals of countable cofinality.
The following fact shows that quite often $Q_\mu$ is not a complete suborder of $Q$,
when $\mu$ varies over ordinals of countable cofinality.
\begin{fact}
Assume $\textrm{cf}(\mu) = \omega$, $\mu \in \omega_2$, 
for some $\beta > \mu$, $\mu$ is a limit point of $C_\beta$ and 
the set of all limit points of $C_\mu$ is cofinal in $\mu$.
Then $Q_\mu$ is not a complete suborder of $Q$.
\end{fact}
\begin{proof}
Assume $\beta > \mu$ such that $\mu$ is a limit point of $C_\beta$ and $\textrm{cf}(\beta) = \omega$.
Let $\nu = \textrm{otp}(C_\beta)$ and  $q = \{(\beta, \{ \nu \})\}$. 
We claim that for all $p \in Q_\mu$ there is an extension $\bar{p} \leq p$ in $Q_\mu$ such that 
$\bar{p}$ is incompatible with $q$.
Fix $p \in Q_\mu$. 
Without loss of generality $\nu \in \bigcup \range(p)$ and $p$ is compatible with $q$.
Let $\xi \in \dom(p)$ such that $\nu \in p(\xi)$. 
Then $p \cup \{(\beta, p(\xi) \cap (\nu +1))\} \in Q$.
Let $\alpha$ be a limit point of $C_\mu$ which is above all elements of $\dom(p)$. 
Then $\bar{p} = p \cup \{ (\alpha, p(\xi) \cap (\nu + 1))\}$ is a condition in $Q_\mu$.
But $\rho(\alpha , \beta)= \textrm{otp}(C_\alpha) < \nu$ and $\nu \in \bar{p}(\alpha)$.
Hence $\bar{p}, q$ are incompatible.
\end{proof}

\begin{lem}\label{finite_mod}
Assume $\mu \in \omega_2$, $x \subset \omega_2$ is finite
and $Q_\mu \lhd Q$. Then $Q_\mu \lhd Q_{\mu \cup x} \lhd Q$.\footnote{$P_0 \lhd P_1$ denotes that $P_0$ is a complete suborder of $P_1$.}
\end{lem}
\begin{proof}
Obviously $Q_\mu \lhd Q$ implies $Q_\mu \lhd Q_{\mu \cup x}$.
It suffices to show  $Q_{\mu \cup x} \lhd Q$ for all $\mu \in \omega_2$ with $Q_\mu \lhd Q$ and finite $x \subset [\mu, \omega_2)$.
Let $q \in Q$. We need to show there is $p \in Q_{\mu \cup x}$ such that every extension $p'$ of $p$ in $Q_{\mu \cup x}$ is compatible 
with $q$. 
Without loss of generality, by extending $q$ if necessary,  we can assume that 
\begin{itemize}
\item $x \subset \dom(q)$,
\item $q$ forces that $\dot{b}_\alpha \wedge \dot{b}_\beta = \max(q(\alpha) \cap q(\beta))$ for all $\alpha, \beta$ in $\dom(q)$,
\item if $\alpha , \beta$ are in $\dom(q)$ and $\nu \in \{0\} \cup \lim(\omega_1)$, then 
$q(\alpha)$ intersects $[\nu, \nu+\omega)$ if and only if $q(\beta)$ does.
\end{itemize}
It is easy to see that if $r \leq_{Q_x}q \rest x$ then $r$ is compatible 
with $q$. 
This already shows that for all finite $y \subset \omega_2$, $Q_y \lhd Q.$
Aside from the extra assumptions on $q$, we can 
also assume that $\mu$ is infinite.\footnote{In fact we can assume that $\mu$ is uncountable. This is because there is no countably infinite $A\subset \omega_2$ with $Q_A \lhd Q$. This is trivial from 
the ideas in the proof of Lemma \ref{finite_mod} and we are not going to use it.}  

Let $p_0 \in Q_\mu$ such that for all extensions $p_1$ of $p_0$ in $Q_\mu$ the conditions $p_1, q$ are compatible.
Since $\mu$ is infinite, by extending $p_0$ if necessary, we can assume 
$\bigcup \range (q) \subset \bigcup \range (p_0)$.
Consequently,
for every 
$\alpha \in \dom(q)$ there is $\alpha' \in \dom(p_0)$ such that 
$q(\alpha) \subset p_0(\alpha')$.
Define $p$ on $\dom(p_0) \cup x$ as follows:
\begin{itemize}
\item If $\alpha \in x$ let $p(\alpha) = p_0(\alpha') \cap (\max(q(\alpha)) +1)$, where $\alpha' \in \dom(p_0)$ such that $q(\alpha) \subset p_0(\alpha')$. 
\item If $\alpha \in \dom(p_0)$ let $p(\alpha) = p_0(\alpha)$.
\end{itemize}
By Condition \ref{initial segment condition} of Definition \ref{Q}, 
for all $\alpha \in x$, $p(\alpha)$ is independent of the 
choice of $\alpha'$. 
Compatibility of $p_0$ and $q$ implies that $p$ is a condition 
in $Q_{\mu \cup x}$.
Let $p'$ be an extension of $p$ in $Q_{\mu \cup x}$.
Then $p' \rest \mu$ is an extension of $p_0$ in $Q_\mu,$ hence it is 
compatible with $q$.
Also the conditions $p' \rest x$ and $q$ are compatible in $Q$, 
since $p' \rest x \leq_{Q_x} q \rest x$.
This means that 
$p'=(p'\rest \mu) \cup (p' \rest x)$
is compatible with $q$ and we are done.
\end{proof}

\section{Climbing Souslin Trees to See $\rho$}
In this section we analyze the external elements of the generic Kurepa tree that is added by the poset $Q_c$.
The aim is to prove Lemma \ref{rho},  which is a general fact about the function $\rho$.
We use Lemma \ref{rho} to find more weakly external elements in the tree which is generic for $Q$.

\begin{prop}\label{limitrepetition}
Fix $\kappa$ a regular cardinal greater than ${(2^{\omega_1})}^+$. 
Assume  $S$ is the set of all $X \in [\omega_2]^\omega$ such that $C_{\alpha_X} \subset X$ and 
$\lim(C_{\alpha_X})$ is cofinal in $X$.
Define 
$\Sigma = \{ M \prec H_\kappa:   M\cap \omega_2 \in S \wedge |L_{\alpha_M}| = \aleph_2 \}.$
Then $\Sigma$ is stationary in $[H_\kappa]^\omega$.
\end{prop}
\begin{proof}
Let $E \subset [H_\kappa]^\omega$ be a club. Fix $\theta$
a regular cardinal above $({2^\kappa})^+$. Let
$\Seq{M_\xi : \xi \in \omega_1}$ be a continuous $\in$-chain of countable elementary submodels of $H_\theta$
such that for all $\xi \in \omega_1$, $M_\xi \cap \omega_2 \in S$ and $M_\xi \cap H_\kappa \in E$. 
Let $\alpha_\xi = \sup(M_\xi \cap \omega_2)$ and 
$\alpha = \sup \{ \alpha_\xi : \xi \in \omega_1 \}$.
By thinning out if necessary, without loss of generality we can assume that for all 
$\xi \in \omega_1$, $\alpha_\xi$ is a limit point of  $C_\alpha$.

Let $f : \{ \eta \in \omega_2 : |L_\eta| \leq \aleph_1 \} \longrightarrow \omega_2$ by $f(\eta) = \sup(L_\eta)$, and $C_f$
be the set of all ordinals that are $f$-closed. Obviously $f \in M_0$ and for all $\xi$, $\alpha_\xi \in C_f$.
But for any $\xi \in \omega_1$, $\sup L_{\alpha_\xi} \notin M_{\xi +1}$.
So, for all $\xi \in \omega_1$, $M_\xi \cap H_\kappa \in E  \cap \Sigma$.
\end{proof}

\begin{lem}
Assume $G \subset Q_c$ is generic and $T$ is the 
Kurepa tree that is added  by $G$. 
Assume $Q/G$ is the quotient poset such that $Q$ is equivalent to $Q_c * (Q/G)$.
For each $\alpha$ of cofinality $\omega_1$, let 
$A_\alpha = \{ x \in T : \exists q \in Q/ G$ $q \Vdash `` x \in \dot{b}_\alpha " \}$.
Then each $A_\alpha$ is a Souslin subtree of $T$. 
Moreover, there is $\alpha \in \omega_2$ of cofinality $\omega_1$ such that  for all $x \in A_\alpha$,
$T_x$ contains $\aleph_2$ many $b_\xi$ with $\textrm{cf}(\xi) = \omega$.
\end{lem}
\begin{proof}
It is trivial that $A_\alpha$ is a Souslin subtree of $T$.
For the rest of the lemma, let $\theta > ({2^{\omega_1}})^+$ be a regular cardinal, and 
assume  $S$ is the set of all $X \in [\omega_2]^\omega$ such that $C_{\alpha_X} \subset X$ and 
$\lim(C_{\alpha_X})$ is cofinal in $X$.
Let $\Seq{M_\xi : \xi \in \omega_1}$ be a continuous $\in$-chain of countable elementary submodels of $H_\theta$
such that for all $\xi \in \omega_1$, $M_\xi \cap \omega_2 \in S$. Let $\alpha_\xi = \sup(M_\xi \cap \omega_2)$ and 
$\alpha = \sup \{ \alpha_\xi : \xi \in \omega_1 \}$.
 Also fix $q \in Q$ with $\alpha \in \dom(q)$, $t \in q(\alpha),$ and  $ \gamma \in \omega_2$. 
 We find $\eta > \gamma$ and $ p \leq q$ such that 
 $\textrm{cf}(\eta) = \omega$, $\eta \in \dom(p),$ and $ t \in p(\eta)$.
 Find $\alpha' \in \lim(C_\alpha)$ such that:
 \begin{enumerate}
 \item $\dom(q) \cap [ \alpha' , \alpha ) = \emptyset$,
 \item $\textrm{otp}(C_{\alpha'})$ is above  all elements of $\bigcup \range(q)$ and all $\rho(\{\alpha , \beta\})$ for  $\beta \in \dom(q)$,  
 \item for all $\beta \in \dom(q) \setminus \alpha,$ $\rho(\alpha, \beta) \leq \rho (\alpha',\beta)$,
 \item $|L_{\alpha'}|= \aleph_2$.
 \end{enumerate}
 Now pick $\eta \in L_{\alpha'}$ which is above $ \gamma$ and all elements of $\dom(q)$ with $\textrm{cf}(\eta)=\omega$. 
 Define $p$ by $\dom(p) = \dom(q) \cup \{ \eta\}$, $q(\zeta) = p(\zeta)$, for all $\zeta \in \dom(q)$ and 
 $p(\eta) = q(\alpha) \cap (t+1)$. 
 We show that for all $\zeta \in \dom(p)$, $\max(p(\zeta) \cap p(\eta)) \leq \rho(\zeta, \eta)$.
 If $\zeta < \alpha'$, 
 $$\max(p(\zeta) \cap p(\eta)) \leq \max (q(\zeta) \cap q(\alpha)) < \rho(\zeta, \alpha) = \rho(\zeta, \alpha')= \rho(\zeta, \eta).$$
 Also $\max (p(\alpha) \cap p(\eta)) = \max (q(\alpha) \cap q(\eta))=t < \textrm{otp}(C_{\alpha'})\leq \rho(\alpha, \eta)$.
 When $\zeta$ is above $\alpha$, 
$$\max(p(\zeta) \cap p(\eta)) \leq \max(q(\zeta) \cap q(\alpha)) < \rho (\alpha, \zeta) \leq \textrm{otp} (C_{\alpha'})\leq \rho(\zeta, \eta).$$
\end{proof}

Now we are ready to prove the main lemma of this section.
\begin{lem}\label{rho}
Let ${(2^{\omega_1})}^+<\kappa_0 <\kappa< \theta$ be regular cardinals such that ${(2^{\kappa_0})}^+< \kappa$, and 
$(2^\kappa)^+ < \theta$.
Let $S$ be the set of all $X \in [\omega_2]^\omega$ such that $C_{\alpha_X} \subset X$ and 
$\lim(C_{\alpha_X})$ is cofinal in $X$.
Assume $\mathcal{A}$ is the set of all countable $N \prec H_\theta$ with the property that if $N \cap \omega_2 \in S$
then there is a club of countable elementary submodels
 $E \subset [H_{\kappa_0}]^\omega$ in $N$ such that for all $M \in E \cap N$,
 $$\rho(\alpha_M, \alpha_N) \leq M\cap \omega_1.$$
 Then $\mathcal{A}$ contains a club.
\end{lem}
\begin{proof}
Assume $G$ is the $\textsc{V}$-generic filter over $Q_c$ and $T$ be the tree that is introduced by $G$.
Assume  $\dot{A}$ is  a $Q_c$-name for an Aronszajn subtree of $T$
with the property that for all $t \in \dot{A}$, the set
$\{ \xi \in \omega_2 : \textrm{cf}(\xi) < \omega_1 \textrm{ and }  t \in b_\xi\}$ has size $\aleph_2$.
Fix $N \prec H_{\theta}$, in $\textsc{V}$, with $\dot{A} \in N$ and $N \cap \omega_2 \in S$.
Suppose for a contradiction that
\begin{center}
$(*)$:for all clubs $E \subset [H_{\kappa_0}]^\omega $ in $N$ there is 
$M \in E \cap N$ such that $\rho(\alpha_M, \alpha_N) > M \cap \omega_1$.
\end{center} 
Let $\delta_M, \delta_N$ be $M \cap \omega_1$ and $N \cap \omega_1$ respectively.
Fix $t \in [\delta_N, \delta_N + \omega)$, $q \in Q_c$ such that $q$ forces that $t \in \dot{A}$.
Obviously, $q$ forces that $t$ is external to $N[\dot{G}]$. In other words, 
$q$ forces that there is a club $E \subset [H_{\kappa_0}[\dot{G}]]^\omega$ in $N[\dot{G}]$
such that for all $Z \in E \cap N[\dot{G}]$, $Z$ does not capture $t$.
Let $\dot{E}$ be a name for the witness $E$ above.
 So $q$ forces that for all $Z \in \dot{E} \cap N[\dot{G}]$, $Z$
does not capture $t$. 
In order to reach a contradiction,
it suffices to show $(*)$ implies that there are $M\prec H_\kappa $ in $N$ and $p \leq q$ in $Q_c$ such that:
\begin{enumerate}
\item $\dot{E} \in M$ and
\item $p$ forces that $M[\dot{G}]$ captures t.
\end{enumerate}
We consider three cases. First, consider the case where $t \notin \bigcup \range(q)$.
Let $\gamma \in (N \cap \omega_2) \setminus \dom(q)$, with $\textrm{cf}(\gamma) = \omega$.
Let $M \prec H_\kappa$ be in $N$ such that $\gamma , \dot{E}$ are in $M$.
Let $p$ be the condition such that $\dom(p) = \dom(q) \cup \{ \gamma \}$, $\forall \xi \in \dom(q)$ $p(\xi) = q(\xi)$,
and $p(\gamma) =  \{t\}$. 
It is obvious that  $p$ is an extension of $q$ and it forces that $M[\dot{G}]$ captures $t$ via $\dot{b}_\gamma$.

Now suppose for some $\xi \in \dom(q) \cap N$, $t \in q(\xi)$. In this case assume 
$M \prec H_\kappa$ is in $N$ such that $\dot{E} , \xi $ are in $M$.
Then $q$ forces that $M[\dot{G}]$ captures $t$ via $\dot{b}_\xi$.

For the last case, suppose $t \in \bigcup \range(q)$ but $\forall \xi \in \dom(q) \cap N$ $t \notin q(\xi)$.
Since any element of $\dot{A}$ is an element of $\aleph_2$ many branches  $b_\xi \subset T$
with $\textrm{cf}(\xi) < \omega_1$, by extending $q$
if necessary, we can assume that there is $\tau \in \dom(q) \setminus \alpha_N$ such that $t \in q(\tau)$.
We consider the partition $\dom(q) = H \cup L \cup R$ where $R = \dom(q) \cap N$ (rudimentary ordinals w.r.t. $N$), 
$L = (\dom(q) \cap \alpha_N ) \setminus R$ (low ordinals), and $H = \dom(q) \setminus \alpha_N$ (high ordinals).
Let $B_t$ be the set of all $\xi \in \dom(q)$ such that $t \in q(\xi)$. So $B_t \cap R = \emptyset$ and $\tau \in B_t$.
By Lemma \ref{limit_in_trace} we have the following about the ordinals in $H$:
\begin{equation}\label{H}
\exists \gamma_0 \in N \cap \omega_2 \textrm{ }\forall \gamma \in N \setminus \gamma_0 \textrm{ }\forall \xi \in H \textrm { }
\rho(\gamma , \xi) \geq \rho (\gamma , \alpha_N) 
\end{equation}
For ordinals in $L$, let $\gamma_1 = \max \{ \min((N \cap \omega_2) \setminus \xi ) :\xi \in L \}$. Then 
\begin{equation}\label{L}
\forall \gamma  \in N \setminus  \gamma_1 \textrm{ }
\forall \xi \in L \textrm{ }
\rho(\xi , \gamma) \geq \delta_N.
\end{equation}
In order to see (\ref{L}), fix $\xi \in L$ and let $\xi' = \min(N \cap \omega_2) \setminus \xi$.
Observe that
$\textrm{cf}(\xi')= \omega_1$.
Let $\gamma \in N$ be above $\gamma_1$. 
We show that $\rho (\xi , \gamma) \geq N \cap \omega_1$.
Note that there is $\alpha \in \xi'$ such that for all $\eta \in (\alpha, \xi')$ the ordinal 
$\textrm{otp}(C_{\xi'} \cap \eta)$ appears in the definition of 
$\rho(\eta, \gamma)$. 
Since $\gamma , \xi' $ are in $N$, by elementarity, the witness $\alpha$ exists in $N$.
Since $\xi \in (\alpha, \xi')$
the ordinal
$\textrm{otp}(C_{\xi'} \cap \xi)$ appears in the definition of $\rho(\xi, \gamma)$.
But $\textrm{otp}(C_{\xi'} \cap \xi) \geq \delta_N$, which shows (\ref{L}). 

Now using $(*)$ choose $M \prec H_\kappa$ in $N$ such that $\rho (\alpha_M, \alpha_N) > \delta_M$ and 
such that
$M$ has $\gamma_0, \gamma_1, R, \bigcup \range(q) \cap N, \dot{E}$ as elements.
Let $\gamma_3 > \max\{ \gamma_0, \gamma_1 \}$ be in $M$ such that   for all 
$ \gamma \in M$ that are above $\gamma_3$, 
$\rho(\gamma , \alpha_N) > \delta_M$. 
The ordinal $\gamma_3$ is guaranteed to exist by Lemma 
\ref{cofinal_sequence}.

For every $\xi \in R$ and  $\eta \in B_t$ by the initial segment 
 requirement on the conditions in $Q$, 
$\max(q(\xi) \cap q(\eta)) = \max (q(\tau) \cap q(\xi))$.
If $\max(q(\xi) \cap q(\tau)) \notin M$ for some $\xi \in R$, we are done.
Assume $\max(q(\xi) \cap q(\tau)) \in M,$ for all $\xi \in R$.
By elementarity, fix $\gamma > \gamma_3$ in $M$ such that $\textrm{cf}(\gamma) = \omega$ and 
\begin{equation}\label{R}
\forall \xi \in R \textrm{ }\rho(\xi , \gamma) > \max (q(\tau) \cap q(\xi)).
\end{equation}
Now define $p \leq q$ as follows:
\begin{itemize}
\item $\dom(p) = \dom(q) \cup \{ \gamma \}$,
\item $ \forall \xi \in \dom(q) \setminus B_t$ $p(\xi) = q(\xi)$,
\item $\forall \xi \in B_t$ $p(\xi ) = q(\xi) \cup \{ \delta_M \}$,
\item $p(\gamma) = p(\tau) \cap (\delta_M + 1)$.
\end{itemize}
Obviously, $p$ forces that $M[\dot{G}]$ captures $t$ via $\dot{b}_\gamma$, 
provided that $p \in Q_c$. 
It is obvious that $p$ fulfills the initial segment requirement. 
Moreover, 
$\bigcup \range (p) \setminus \bigcup \range (q) = \{ \delta_M \}$
because $\bigcup \range (q) \cap N \in M$.
We show for all 
$\xi, \eta$ in $\dom(p)$, $\rho\{ \xi, \eta \} > \max(p(\xi) \cap p(\eta)).$ This can be done by managing the following six cases.

First assume that $\xi, \eta$ are in $\dom(q)$ and at least one of them
is not in $B_t$.
Equivalently, $\eta \in \dom(q)$ and  $\xi \in \dom(q)  \setminus B_t$. 
Then $\delta_M \notin p(\xi)$ and $p(\xi) = q(\xi)$.
Hence 
$\max (p(\xi) \cap p(\eta)) = \max (q(\xi) \cap q(\eta))< \rho\{\xi, \eta \}$ because $q$ is a condition in $Q$.

For the second case assume $\xi, \eta$ are both in $B_t$. 
Recall $\delta_m < t$ and $ t \in q(\xi) \cap q(\eta)$. Then 
$\max(p(\xi) \cap p(\eta)) = \max (q(\xi) \cap q(\eta)) < \rho \{ \xi, \eta \}$.
So far we have shown that condition \ref{delta condition} of Definition 
\ref{Q} holds for pairs of ordinals in $\dom(q)$.

For the fourth case assume $\xi \in H$ and $\eta = \gamma$. 
The way we chose $\gamma_3$, 
and  (\ref{H}) guarantees that  
$\rho(\gamma, \xi) \geq \rho(\gamma , \alpha_N) > \delta_M = \max(p(\gamma)).$

For the fifth case assume $\xi \in L$ and $\eta = \gamma$. 
Then  (\ref{L}) implies that 
$\rho(\xi , \gamma) \geq \delta_N > \max(p(\gamma))$.

For the sixth case assume $\xi \in R$ and $\eta = \gamma$. 
Then (\ref{R}) implies that 
$\rho(\xi, \gamma) > \max(q(\tau) \cap q(\xi)) = \max (p(\gamma) \cap p(\xi))$.
 Therefore, $p \in Q_c$.
\end{proof} 

\section{$\rho$ Introduces Aronszajn Subtrees Everywhere in $T$ }
In this section we will use Lemma \ref{rho} to show that every Kurepa subset of the generic Kurepa 
tree has an Aronszajn subtree. 
Here a subset $Y$ of $T$ is said to be a \emph{Kurepa subset} if it is a Kurepa tree when it is considered with the order inherited from $T$. Note that $Y$ is not necessarily downward closed. 
The theorems in this section are not using any large cardinal assumption.

\begin{lem}\label{no_extera_branches}
Assume 
$X \subset  \omega_2$ is uncountable,
$Q_X \lhd Q$ and
$T$ is the generic tree for $Q_X$. 
Then $\{ b_\xi : \xi \in X \}$ is the set of all cofinal branches of $T$ in the forcing extension by $Q_X$.
\end{lem}
\begin{proof}
Assume $P= Q_X$ and $\pi$ is a $P$-name for a branch that is different from all $b_\xi$, $\xi \in X$. 
Inductively construct a sequence $\Seq{p_\eta: \eta \in \omega_1 }$ as follows. 
The condition  $p_0 \in P$ is arbitrary. 
 If $\Seq{p_\eta : \eta < \alpha }$ is given, find $p_\alpha \in P$ such that:
\begin{itemize}
\item $p_\alpha$ decides $\min (\pi \setminus \bigcup 
\{ b_\xi : \xi \in \bigcup \{ \dom(p_\eta) : \eta \in \alpha \} \})$ to be $t_{p_\alpha}$,
\item $t_{p_\alpha} \in \bigcup \range(p_\alpha)$,
\item for every $\beta \in \dom(p_\alpha)$, $\Ht(\max (p_\alpha(\beta))) > \Ht (t_{p_\alpha}).$\footnote{Note that the levels of the generic tree are in the ground model.}
\end{itemize}
Let  $A= \{ p_\alpha : \alpha \in  \omega_1 \}$. 
By going to a subset of $A$ if necessary,
we may assume that $\{\dom(p) : p \in A\}$ forms a $\Delta$-system with 
root $d$. Also $\{\bigcup\range(p):p \in A \}$ forms a $\Delta$-system with root $c$.
Moreover, we may assume that elements of $A$ are pairwise isomorphic structures 
and the isomorphism between them fixes the root.
By Lemma \ref{cc} there is an uncountable set $B \subset A$ such that  for every $p,q$ in $B$
if $\alpha \in \dom(p) \setminus \dom(q)$,
$\beta \in \dom(q) \setminus \dom(p)$, and
 $\gamma \in d$, then 
 \begin{itemize}
 \item $\rho\{\alpha , \beta\} > \max(c)$ and 
 \item $\rho \{\alpha , \beta\} \geq \min \{ \rho\{\gamma, \alpha\}, \rho\{\gamma, \beta\}\}$.
 \end{itemize}
Note that for all $p \in B$, $\bigcup \range(p) \subset \omega_1$. 
So without loss of generality, by replacing $B$ with an uncountable subset if necessary, 
we can assume the following: Whenever $p, q$ are in $B$ either 
\begin{itemize}
\item $c < a = \bigcup \range(p) \setminus c < b = \bigcup \range(q) \setminus c$ or 
\item $c  < b = \bigcup \range(q) \setminus c < a = \bigcup \range(p) \setminus c$.
\end{itemize}

We claim that the elements of $B$ are pairwise compatible.
In order to see this, fix $p,q$ in $B$. 
By symmetry, we can assume that 
$$c < a = \bigcup \range(p) \setminus c < b = \bigcup \range(q) \setminus c.$$
We define the common extension $r$ of $p,q$ on $\dom(p) \cup \dom(q)$
as follows:
For $\gamma \in d$ let $r(\gamma) = p(\gamma) \cup q(\gamma)$, and
for $\alpha \in \dom(p) \setminus \dom(q)$ let $r(\alpha) = p(\alpha)$.
For $\beta \in \dom(q) \setminus \dom(p)$ we have two cases. Either for all 
$\gamma \in d$, $\max(q(\gamma) \cap q(\beta)) \in c$ or 
there is a unique $\gamma \in d$ such that $\max(q(\gamma) \cap q(\beta)) \in b$. 
In the first case let $r(\beta)= q(\beta)$ and in the second case let 
$r(\beta) =p(\gamma) \cup q(\beta)$.
In order to see that there is no possibility outside of these two cases, assume for a contradiction that 
$\gamma_0 , \gamma_1$ are in $d$ and for $i \in 2$,
$\max(q(\gamma_i) \cap q(\beta)) \in b \setminus c$.
In other words, both $q(\gamma_0), q(\gamma_1)$ intersect $q(\beta)$ above $\max(c)$.
So there is $\nu \in b \setminus c$ such that $\nu \in q(\gamma_0) \cap q( \gamma_1)$.
Recall that the elements of $B$ are isomorphic structures via the isomorphisms which fix the roots.
Therefore, for each $s \in B$ there is $\nu_s \in \bigcup \range(s) \setminus c$ such that 
$\nu_s \in s(\gamma_0) \cap s(\gamma_1).$
But this contradicts the fact that $\rho(\gamma_0, \gamma_1)$ is countable,
since $\{ \bigcup \range (s) : s \in B \}$ is an uncountable $\Delta$-system with root $c$.

First we will show that $r$ satisfies Condition \ref{initial segment condition}.
Note that if $\gamma_1, \gamma_2$ are both in $d$ then
$p(\gamma_1) \cap p(\gamma_2) \subset c$ and $q(\gamma_1) \cap q(\gamma_2) \subset c$.
In order to see this, assume this is not the case. 
Then 
by the fact that the conditions in $B$ are pairwise isomorphic,
$\sup \{\max (s(\gamma_1) \cap s(\gamma_2)): s \in B \} = \omega_1$ which implies that 
$\rho(\gamma_1, \gamma_2) \geq \omega_1$. But this is absurd.
Now assume $i \in (p(\gamma_1) \cup q(\gamma_1)) \cap (p(\gamma_2)\cup q(\gamma_2))$,
$j < i$ and $j \in (p(\gamma_1) \cup q(\gamma_1))$.
We will show that $j \in  p(\gamma_2)\cup q(\gamma_2)$.
Note that $j \in c$. 
Then $j \in p(\gamma_1) \cap c = q(\gamma_1) \cap c$. 
Since $p,q$ both satisfy Condition \ref{initial segment condition} and $i \in p(\gamma_2)\cup 
q(\gamma_2) $, we have $j \in p(\gamma_2)\cup q(\gamma_2)$.
If $\alpha \in \dom(p)  \setminus \dom(q)$ and $\gamma \in d$ note that 
$$r(\alpha) \cap r(\gamma)= p(\alpha)\cap (p(\gamma) \cup q(\gamma)) = p(\alpha)\cap p(\gamma). $$
But  $p(\alpha)\cap p(\gamma)$ is an initial segment of both $p(\alpha)$ and $r(\gamma)$ because $a<b$. 
If $\beta \in \dom(q) \setminus \dom(p)$ and for all  
$\gamma \in d$, $\max(q(\gamma) \cap q(\beta)) \in c$
the argument is the same. 
So assume that for a unique $\gamma_\beta \in d$, 
$\max(q(\beta) \cap q(\gamma_\beta)) \in b$.
Then it is easy to see that
 $r(\beta) \cap r(\gamma_\beta)= p(\gamma_\beta) \cup (q(\beta) \cap q(\gamma_\beta))$
 is an initial segment of both $r(\beta), r(\gamma_\beta)$.
 If $\beta \in \dom(q) \setminus \dom(p)$ and $\gamma \in d \setminus \{\gamma_\beta \}$,
 in order to see $r(\beta)\cap r(\gamma)$ is an initial segment of both $r(\beta), r(\gamma)$, 
 note that 
$$r(\beta) \cap r(\gamma)=(p(\gamma_\beta) \cup q(\beta))\cap (p(\gamma) \cup q(\gamma)) \subset c.$$
Then $r(\beta) \cap r(\gamma) = p(\gamma_\beta) \cap p(\gamma)$ which makes Condition \ref{initial segment condition} trivial.
We leave the rest of the cases to the reader.

For Condition \ref{delta condition}, we only verify the case $\alpha \in \dom(p) \setminus \dom(q)$
and $\beta \in \dom(q) \setminus \dom(p) $. 
If for all $\gamma \in d$, $\max(q(\gamma) \cap q(\beta)) \in c$, there is nothing to prove.
Assume for some unique $\gamma \in d$, $\max(q(\gamma) \cap q(\beta)) \in b$. 
Obviously, $r(\alpha) \cap r(\beta)= (p(\alpha) \cap p(\gamma)) \cup (p(\alpha) \cap q(\beta))$.
But $\max (p(\alpha) \cap q(\beta)) \leq \max(c) < \rho \{\alpha , \beta \}$. Moreover, 
$$\max (p(\alpha) \cap p(\gamma)) \leq \max (a) < \min (b) \leq \max (q(\beta) \cap q(\gamma))
\leq \rho\{\gamma, \beta\}.$$
This means that $\max (p(\alpha) \cap p(\gamma)) < \min \{ \rho\{\alpha, \gamma \} , \rho \{ \beta, \gamma\} \} \leq \rho \{ \alpha , \beta \}.$
Therefore, $\max (r(\alpha) \cap r(\beta)) < \rho \{ \alpha , \beta \}$.

We have two possible cases: either there is an uncountable $C \subset B$ such that for all $p \in C$, 
there is $\gamma \in d$ with $t_p \in p(\gamma)$, or there are only countably many $p \in B$ such that
for some $\gamma \in d$, $t_p \in p(\gamma).$
If such an uncountable $C$ exists, let $s \in P$
such that $s$ forces that the generic filter intersects $C$ on an uncountable set.
Then for some $\gamma \in d$, $s \Vdash |\pi \cap b_\gamma| = \aleph_1$. 
But this contradicts the fact that $\pi$ was a name for a branch that is different from all $b_\xi$'s.

Now assume that there is a countable set $D \subset B$
such that if $p \in B$ and for some $ \gamma \in d$, $t_p \in p(\gamma)$ then $p \in D$.
We can choose $p,q$ in $B \setminus D$ such that:
\begin{enumerate}
\item for some $ \alpha  \in \dom(p) \setminus \dom(q)$, $t_p \in p(\alpha),$ 
\item for some $\beta \in \dom(q) \setminus \dom(p)$, $t_q \in q(\beta)$, 
\item $p$ forces that $t_p$ is not in the branches that are indexed by the ordinals in $d$, and 
\item $\max (c) + \omega < t_p$ and $ t_p + \omega < t_q$.
\end{enumerate}
Obviously, (1), (2) are automatically true for any $p,q$ in $B \setminus D$.
We claim that there is at most one $p_\eta \in B$ which does not force that $t_{p_\eta}$
is not in the branches that are indexed by the ordinals in $d$.
In order to see this, assume for a contradiction that $\zeta < \eta < \omega_1$ and 
$p_\eta , p_\zeta $ are counterexamples to our claim.
Then $p_\eta$ decides  $\min (\pi \setminus \bigcup 
\{ b_i : i \in \bigcup \{ \dom(p_j) : j \in \eta \} \})$ to be $t_{p_\eta}$.
In particular, $p_\eta$ forces that $t_{p_\eta} \notin \bigcup\{b_i: i \in \dom (p_\zeta) \}
\supset \bigcup\{b_i: i \in d \}$. 
Therefore, $p_\eta$ satisfies Condition (3), which is a contradiction.
By the same argument, if $p \neq q$ are in $B$ then $t_p \neq t_q$.
Therefore, it is easy to choose $p,q$ in $B$ such that the four conditions above hold.

Let $a,b,c,d$ be as above.
We will find a common extension of $p,q$ which forces that $t_p$ is not below $t_q$.
This contradicts the fact that $\pi$ was a name for a branch.

First consider the case in which for all $\gamma \in d$,  $\max(q(\beta) \cap q(\gamma)) \in c$. 
Let $r$ be the common extension of $p,q$ described as above.
Recall that $r(\beta) = q(\beta )$ in this case. 
Let $\xi \in (t_p, t_p + \omega) \setminus (a \cup b)$.
Note that $\xi > \max(c)$. Let $X = \{ \eta \in \dom (r) : \max(r(\beta) \cap r(\eta)) > \xi \}$.
Obviously, $X \cap \dom(p) = \emptyset$ and $\beta \in X$.
Extend $r$ to $r'$ such that $\dom(r') = \dom(r)$, $r$ and $ r'$ agree on any element of their domain 
which is not in $X$, and $r'(\eta) = r(\eta) \cup \{ \xi\}$ for all $\eta \in X$.
Checking  $r'$ is a condition is routine.
The condition $r'$  forces that in the generic tree $\Ht(\xi) = \Ht(t_p)$ and they are distinct.
Therefore, it forces that $\xi < t_q$ and that $t_p$ is not below $t_q$.

Now assume for some $\gamma \in d$, $\max(q(\beta) \cap q(\gamma)) \in b$.
Again assume that $r$ is the common extension described above.
So $r(\beta) = p(\gamma) \cup q(\beta)$, and $r$ forces that 
$\max(p(\gamma))$ is below $t_q$ in the generic tree.
Recall that $\Ht(\max(p(\gamma))) > \Ht(t_p)$ and $p$ forces that $t_p$
is not in the branches indexed by the ordinals in the root $d$.
Hence $p$ forces that $t_p$ is not below $\max(p(\gamma))$.
Since $r \leq p$, it forces that $t_p$ is not below $t_q$ in the generic tree.
\end{proof}
Now we are ready to prove the main theorem of this section.
\begin{thm}\label{AinK}
It is consistent that there is a Kurepa tree $T$
such that every Kurepa subset of $T$ has an Aronszajn subtree.
\end{thm}
\begin{proof}
Assume $G$ is a generic filter for the forcing $Q$, and $T$ is the tree introduced by $G$.
Since $Q$ is ccc, it preserves all cardinals and $T$ is a Kurepa tree.

Assume $U$ is a Kurepa subset of $T$, and $X$ is the set of all  $\xi \in \omega_2$ such that
 $ b_\xi \cap U$ is uncountable.
Let $\Seq{N_\xi : \xi \in \omega_1}$ be a continuous $\in$-chain of countable elementary submodels of 
$H_\theta$ such that $U \in N_0$ and
for all $\xi \in \omega_1,$ $N_\xi \in \mathcal{A}$,
where $\mathcal{A}$ is the same club as in Lemma \ref{rho}.
Let $ N_{\omega_1} = \bigcup\limits_{\xi \in \omega_1} N_\xi$, 
 $\mu = N_{\omega_1} \cap \omega_2$. Fix $\eta \in X$ above $\mu$.
By Proposition \ref{Asubtree},
it suffices to show that for some $\xi \in \omega_1$, 
the first element of  $b_\eta \cap U $ whose height is more than $N_\xi \cap \omega_1$ 
is weakly external to $N_\xi$
witnessed by  some stationary set $\Sigma$.

Without loss of generality we can assume that for all $\xi \in \omega_1$:
\begin{itemize}
\item $\alpha_{N_\xi} = \sup(N_\xi \cap \omega_2)$ is a limit point of $C_\mu$, 
\item $N_\xi \cap \omega_2 \supset C_{\alpha_{N_\xi}}$ and
\item $\lim (C_{\alpha_{N_\xi}})$ is a cofinal in $\alpha_{N_\xi}$.
\end{itemize}

In order to see this, let $f$ from $\omega_1$ to $\mu$  be the function which is defined as follows: For each 
$\xi \in \omega_1$, $f(\xi)$ is the least $\zeta \in \omega_1$ with $N_\zeta \supset C_\mu \cap \alpha_{N_\xi}$.
Now observe that if $\xi$ is $f$-closed then it satisfies the second condition.
For the other two conditions, note that the sets 
$\{ \alpha_{N_\xi} : \xi \in \omega_1 \}$ and the set of all $\gamma \in C_\mu$ which
are limit of limit points in $C_\mu$ are clubs in $\mu$.

Let $\xi \in \omega_1$ be such that $\textrm{otp}(C_{\alpha_{N_\xi}}) > \rho(\mu, \eta)$
and for all $\zeta > \xi$, $\rho(\alpha_{N_\zeta}, \eta) > \rho (\mu, \eta)$.
Then note that $\rho(\mu, \eta) \in N_\xi$.
Use Lemma \ref{rho} to find $E\in N_\xi$ which is  a club of countable elementary submodels of 
$H_{\omega_3}$ such that for all $M \in E \cap N_\xi$, $\rho(\mu , \eta) \in M$ and 
$\rho(\alpha_M, \alpha_{N_\xi}) \leq M\cap \omega_1.$
Now let  $\Sigma$ be the set of all  $M \in E$ such that $M \cap \omega_2 \supset C_{\alpha_M}$ and $\lim (C_{\alpha_M})$ is a cofinal subset of $\alpha_M$.
Obviously, $\Sigma$ is stationary and in $N_\xi$.
Let $M \in \Sigma \cap N_\xi$. 
We want to show that $M$ does not capture $b_\eta,$ as a branch of $T$. Equivalently,
for all $b \in  M$ which is a cofinal branch of $T$, $\Delta (b, b_\eta) \in M$.
By the lemma above, it suffices to show that for all $\gamma \in M$, $\rho(\gamma, \eta) \leq M \cap \omega_1$.
Recall that:
$$\rho(\gamma, \eta) \leq \max\{ \rho(\gamma , \alpha_M), \rho(\alpha_M, \mu), 
\rho(
\mu, \eta) \}.$$
Fix $\beta$ which is a limit point of $C_{\alpha_M}$ and which is above $\gamma$. 
Since $\beta\in M$ and $\rho(\gamma , \beta) = \rho(\gamma, \alpha_M)$, we have that  
$$\rho(\gamma , \alpha_M) \in M.$$ 
Since $M \in E$ and $\alpha_{N_\xi} \in \lim(C_\mu)$, we obtain
$$\rho (\alpha_M, \mu) = \rho(\alpha_M, \alpha_{N_\xi}) \leq M\cap \omega_1.$$
Recall that $\rho(\mu , \eta) \in M$. 
Therefore, $\rho(\gamma  , \eta) \leq M \cap \omega_1$.

Now assume  $M \in \Sigma \cap N_\xi$, $t$ is the first element of 
$b_\eta$ whose height is more than $N_\xi \cap \omega_1$.
It suffices to show that $M$ does not capture $t$ as an element in $U$.
Assume $b \subset U$ is a cofinal branch of $U$ which is in $M$ and 
$b$ contains $\{s \in U \cap M: s <t \}$.
Since $t \notin M,$ the set $\{s \in U \cap M: s <t \}$ has order type $M \cap \omega_1$.
Let $b_\gamma$ be the downward closure of $b$ in $T$. 
Then obviously $\gamma \in M$.
But then the order type of $b_\gamma \cap b_\eta$ is at least $M \cap \omega_1$,
which is a contradiction.
\end{proof}
We finish this section by a corollary which relates the theorem above to Martin's Axiom.
\begin{cor}
Assume $\MA_{\omega_2}$ holds and 
$\omega_2$ is not a Mahlo cardinal in $\textsc{L}$.
Then there is a Kurepa tree with the property that every Kurepa subset has an
Aronszajn subtree.
\end{cor}

\section{Taking Komjath's Inaccessible Away}

In this section we will show that if there is an inaccessible cardinal in $\textsc{L}$ then there is a
model of $\ZFC$ in which every Kurepa tree has an Aronszajn subtree.
We will be using the following notation. 
If $G \subset Q$ is a generic filter and $X \subset \omega_2$ with $Q_X \lhd Q$, we use 
$G_X$ in order to refer to $G \cap Q_X$.
If $X \subset A$ and $Q_X \lhd Q_A \lhd Q$,
$R_{X , A}$ refers to the ccc poset such that 
$Q_A= Q_X *\dot{ R}_{X, A}$. 
Note that the generic tree $T$ is in $\textsc{V}[G_X]$
if $X$ is uncountable and $Q_X \lhd Q$.
Then $R_{X, A}$ can be described 
more explicitly in the forcing extension by $G_X$ as follows.
Let $T$ be the generic tree for $Q_X$ and $b_\xi$ be the set of all $ t \in T$ such that $t \in q(\xi)$ for some $ q \in G_X $.
Recall that $b_\xi$ is an uncountable downward closed branch of $T$. 
Moreover, every branch of $T$ in the forcing extension by $G_X$ has to be $b_\xi$ for some $\xi \in X$.
The poset $R_{X, A}$ consists of finite partial functions $p$
from $A \setminus X$ to $T$ such that:
\begin{enumerate}
\item for every $\alpha \in \dom(p)$ and  $\xi \in X$,  
$(p(\alpha) \wedge b_\xi) < \rho\{\xi, \alpha\}$ and
\item for all $\alpha < \beta$ in $\dom(p)$, $(p(\alpha) \wedge p(\beta)) < \rho(\alpha, \beta)$.
\end{enumerate}
In $R_{X, A}$, $q \leq p$ if $\dom(q) \supset \dom(p)$ and 
$p(\alpha) \leq_T q(\alpha)$  for all $\alpha \in \dom(p)$.
We sometimes use the notation $R_A(B)$ in order to refer to 
$R_{A, A \cup B}$ if $A, B$ are disjoint.

Also, for finite $x \subset [\mu, \omega_2)$, let 
$S^\mu[x]$ be the set of all $ \Seq{v_i: i \in |x|} \in T^{[|x|]}$
such that for some $q \in R_{\mu, \omega_2}$:
\begin{itemize}
\item $\dom(q) \supset x$ and
\item for all $i \in |x|$,    $q(x(i))= v_i$.
\end{itemize}
So in particular every condition in $R_{\mu, \omega_2}$ force that 
$\bigotimes\limits_{\alpha \in x} \dot{b}_\alpha \subset S^\mu[x]$. 
For $\alpha \in \omega_2 \setminus \mu$, 
we use $S^\mu[\alpha]$ instead of $S^\mu[\{\alpha \}]$.

\begin{lem}\label{finite}
Assume 
$\omega_1<\mu < \omega_2$, $Q_\mu \lhd Q$ in $\textsc{V}$
and $G \subset Q$ is $\textsc{V}$-generic.
Let $K$ be an $\omega_1$-tree in $\textsc{V}[G_\mu]$
and $b \subset K$ be  a cofinal branch in $\textsc{V}[G]$.
Then there is a finite $x \subset [\mu , \omega_2)$ such that 
$b \in \textsc{V}[G_{\mu \cup x}]$.
\end{lem}
\begin{proof}
Work in $\textsc{V}[G_\mu]$.
Let $T$ be the generic tree that is introduced by $G_\mu$, 
$r \in R_{\mu,  {\omega_2}} \cap G$,
$\tau \subset K \times \{ q \in R_{\mu,  {\omega_2}} : q \leq r \}$ be an 
$R_{\mu,  {\omega_2}}$-name. Assume  
$r \Vdash_{R_{\mu, \omega_2}} ``\tau$ is a cofinal branch of $K$ which is not in 
$\textsc{V}[G_{\mu } * \dot{H}_x]$ for all finite $x \subset [\mu, \omega_2)"$, where $\dot{H}_x$ is the canonical name for 
the $\textsc{V}[G_\mu]$-generic filter of $R_\mu(x)$.
For every $u \in K$, let $C_u=\{q\leq r: q \Vdash_{R_{\mu, \omega_2}} u \in \tau \}$ and let $E_u \subset C_u$  such that:

\begin{enumerate}
\item $E_u$ is an antichain that is maximal in $C_u$.
\item If $q \in E_u$ and $\alpha \in \dom(q)$
then $\Ht_T (q(\alpha)) \geq \Ht_K(u)$.
\item If $q \in E_u$  then $q$ is a one-to-one function whose range consists of the elements of the same height in $T$.

\end{enumerate}  
$R_{\mu, \omega_2}$ is ccc so $E_u$ is countable. 
Let
$\tau'=\{\{u \} \times C_u : u \in K \}$ and
 $\tau'' = \bigcup \{ \{u\} \times E_u: u \in K \}$.
Observe that 
$r \Vdash \tau = \tau' = \tau''$.
Without loss of generality, we assume $\tau = \tau''$, or in other 
words   $\tau[\{u \}] = E_u$      for all $u \in K$.
Let $U$ be the set of all $u \in K$ such that for some $q \leq r$ in 
$R_{\mu, \omega_2}$ the condition $q$ forces that $u \in \tau$.
$R_{\mu, \omega_2}$ is ccc, so
$U$  is a Souslin tree in $\textsc{V}[G_\mu]$.

Let $\Gamma \subset \range(\tau)$ be uncountable such that 
$\{ \dom(p) : p \in \Gamma \}$ forms a $\Delta$-system with root $w$. 
By thinning $\Gamma$ out if necessary, we can assume  
the conditions in $\Gamma$ have the same cardinality $k+ |w| \in \omega$. Note that $r \in R_{\mu}(w)$. 
Observe that $U = \dom(\tau'') = \dom (\tau)$.
Consequently, if $x \subset \omega_2$ is finite such that 
$\{ u \in K : \exists p \in G_{\mu \cup x} \textrm{  } p \Vdash u \in \tau \}$
is a cofinal branch of $K$
then $U$ is not Aronszajn in $\textsc{V}[G_{\mu \cup x}]$.
In order to reach a contradiction, assume in $\textsc{V}[G_\mu]$ that 
\begin{equation}\label{Xi}
r \Vdash_{R_{\mu}(w)}
``U \textrm{ is Aronszajn.}"
\end{equation}

By the second assumption on $E_u$'s, the set $\Gamma_w =\{ p \rest w : p \in \Gamma \}$
is an uncountable subset of $R_\mu(w)$ consisting of conditions 
extending $r$. 
Therefore $R_\mu (w)$ has  an extension $r'$ of $ r$ which forces that
$\Gamma_w \cap \dot{H}_w$ is uncountable.
In order to contradict (\ref{Xi}), we need to work in $\textsc{V}[G_\mu * H_w]$
and some specific forcing extensions of this model. 
Here $H_w \subset R_\mu (w)$ is a $\textsc{V}[G_\mu]$-generic filter which contains $r'$ and consequently intersects 
$\Gamma_w$ on an uncountable set.
Due to similarity of arguments and for easier notation let's assume 
$\Gamma_w \cap G_{ \mu \cup w}$ is uncountable 
and work with $\textsc{V}[G_{\mu \cup w}]$ instead of $\textsc{V}[G_\mu * H_w]$.

Let $ \Acal \in \textsc{V}[G_{\mu \cup w}]$ be the set of all $ p \rest (\dom(p) \setminus w)$ such that 
$p \in \Gamma$ and $p \rest w \in G_{\mu \cup w}$.
By the third assumption on $E_u$'s, 
if $p \in \Gamma$ and $p\rest (\dom(p) \setminus w) \in \Acal$
then $p$ is compatible with every condition in $G_{\mu \cup w}$.
We are done if $\Acal$ is countable. 
In order to see this, assume $\Acal$ is countable. Then there is an uncountable 
$\Gamma' \subset \Gamma$ in $\textsc{V}[G_{\mu \cup w}]$ 
such that: 
\begin{itemize}
\item for all $p \in \Gamma'$, $p \rest w \in G_{\mu \cup w}$ and 
\item for all $p,q$ in $\Gamma'$, $p \rest (\dom(p) \setminus w) = q \rest (\dom(q) \setminus w)$.
\end{itemize} 
In particular $\Gamma'$ is an uncountable collection of pairwise compatible conditions in $\Gamma$. 
But then $\Gamma'$ defines an uncountable branch of $U$ in 
$\textsc{V}[G_{\mu \cup w}]$ which already contradicts (\ref{Xi}).

From now on, assume $\Acal$ is uncountable.
Using forcings described in Lemma \ref{AddAntichain} in finitely many steps, 
we will find an Aronszajn tree preserving and $\omega_2$-preserving forcing extension of 
$\textsc{V}[G_{\mu \cup w}]$
which has an uncountable $\Acal' \subset \Acal$
consisting of pairwise compatible conditions.
But $G_{\mu \cup w} , \Acal', \tau$ together 
can define an uncountable branch in $U$.
This contradicts (\ref{Xi}) because our forcing extension was Aronszajn tree preserving and 
was supposed to keep $U$ Aronszajn.
It is worth noting  that the work behind finding $\Acal'$ requires working with the 
$\rho$-function in the forcing extension. This is where we need 
our forcing extension to preserve $\omega_2$.

For each $p \in \Acal$, let $d_p: k \longrightarrow \dom(p)$ 
be the unique strictly increasing bijection.
Let $\Seq{I_l: 0 < l \leq \frac{k(k+1)}{2}+1}$
be a sequence listing all $I \subset k$ with $0< |I|\leq 2$
such that all singletons are listed before pairs. 
We are going to find  $\Seq{\textsc{V}_l, \Acal_l: l \leq \frac{k(k+1)}{2}+1}$, by induction on $l$, such that: 
\begin{itemize}
\item $\textsc{V}_0 = \textsc{V}[G_{\mu \cup w}]$.
\item $\textsc{V}_{l+1}$ is an Aronszajn tree preserving and $\omega_2$-preserving forcing 
extension of $\textsc{V}_l$.
\item $\Acal_l \in \textsc{V}_l$ is uncountable for all $l$.
\item $\Acal_{l+1} \subset \Acal_l \subset \Acal_0 = \Acal$.
\item If $\{p,q\} \subset \Acal_l$ then $p\rest \{ d_p(n): n \in I_l\}$ and $q \rest \{ d_q(n) : n \in I_l \}$ 
are compatible in $R_{\mu \cup w,  {\omega_2}}$.
\end{itemize}
We proceed by finding $\textsc{V}_l, \Acal_l$ when
$\textsc{V}_{l-1}, \Acal_{l-1} , I_l$ are given.
First assume $0< l \leq k$, which means  $I_l = \{n \}$ for some 
$n \in k$. 
This task can be done by managing the following cases:
\begin{enumerate}
\item The map $p \mapsto p(d_p(n))$ is constant on some uncountable 
subset of $\Acal_{l-1}$.
\item The map $p \mapsto p(d_p(n))$ is countable-to-one and the downward closure of $\{p(d_p(n)): p \in \Acal_{l-1} \}$ has an uncountable branch.
\item The map $p \mapsto p(d_p(n))$ is countable-to-one and the downward closure of $\{p(d_p(n)): p \in \Acal_{l-1} \}$ is Aronszajn.
\end{enumerate}
For the first case, fix uncountable $\Bcal \subset \Acal_{l-1}$ such that 
$p \mapsto p(d_p(n))$ is constant on $\Bcal$.
Let $\nu = p(d_p(n))$ for some (any) $p \in \Bcal$.
Let $\Acal_l \subset \Bcal $ be uncountable such that if  $p \neq q$ 
are  in $\Acal_{l}$ then  $\rho\{d_p(n), d_q(n) \} > \nu$.
It is easy to see that  
$\Acal_l$ together with $\textsc{V}_l = \textsc{V}_{l-1}$  works.

For the second case let 
$W$ be the downward closure of  $\{p(d_p(n)): p \in \Acal_{l-1} \}$ in $T$.
By Lemmas \ref{AddAntichain} and \ref{no_extera_branches},
let $\xi \in \mu \cup w$ such that $b_\xi \subset W$.
Let $\Seq{p_i : i \in \omega_1}$ be a sequence in $\Acal_{l-1}$
such that $\Seq{p_i (d_{p_i}(n)) \wedge b_\xi: i \in \omega_1}$
is strictly increasing.
Let $\Gamma_0 \subset \omega_1$ be uncountable such that 
 $\Seq{\alpha_i = d_{p_i}(n): i \in \Gamma_0}$ and 
 $\Seq{\rho\{\alpha_i, \xi \} : i \in \Gamma_0}$
are both strictly increasing.
Recall that $\rho \{ \alpha_i , \xi \} \geq b_\xi \wedge p_i(\alpha_i)$, so this is possible.
Find uncountable $\Gamma_1 \subset \Gamma_0$ such that 
$\rho(\alpha_i, \alpha_j) \geq \min \{ \rho\{ \alpha_i, \xi \}, \rho\{ \alpha_j, \xi\} \}$
for $i < j$ in $\Gamma_1$.
In order to see $\Acal_l = \{ p_i : i \in \Gamma_1 \}$ and 
$\textsc{V}_l = \textsc{V}_{l-1}$ work, assume for a contradiction 
that $p_i (\alpha_i) \wedge p_j (\alpha_j)  \geq  \rho ( \alpha_i , \alpha_j)$
for some $i < j$  in $\Gamma_1$. Then 
$$\rho\{ \xi , \alpha_i \} >  p_i(\alpha_i) \wedge b_\xi = p_i(\alpha_i) \wedge p_j(\alpha_j) \geq \rho(\alpha_i, \alpha_j) \geq \rho\{\alpha_i, \xi \}, $$
which obviously is a contradiction.

For the third case, let $W$ be a pruned downward closed uncountable 
subtree of the downward closure of $\{p(d_p(n)): p \in \Acal_{l-1} \}$ in $T$.
Let $\textsc{V}_l$ be a forcing extension 
of $\textsc{V}_{l-1}$ which preserves Aronszajn trees and $\omega_2$ and which adds an uncountable 
antichain $A \subset W$. 
From now on we work in $\textsc{V}_{l}$.
Fix $\gamma > \sup \{ d_p(n) : p \in \Acal_{l-1} \}$ in $\omega_2$ and $\Seq{t_i : i \in \omega_1}$ in $A$ such that if $i<j$ then $\Ht(t_i) < \Ht(t_j)$.
Since $W$ is pruned,  for every $t \in W$ there are 
uncountably many $p$ in $ \Acal_{l-1}$ with $t \leq_T p(d_p(n))$.
Since $\omega_2$ is preserved, the square sequence of $\textsc{V}_{l-1}$ is a square sequence in $\textsc{V}_l$.
Therefore,
for each $i \in \omega_1$  there is $p_i \in \Acal_{l-1}$ such that $t_i \in \rho(d_{p_i}(n), \gamma)$ and $t_i <_T p_i(d_{p_i}(n))$. Let $\alpha_i = d_{p_i}(n)$.
Find uncountable $\Gamma_0 \subset \omega_1$ such that $\Seq{\alpha_i: i \in \Gamma_0}$ and $\Seq{\rho(\alpha_i, \gamma): i \in \Gamma_0}$ are both strictly increasing.
Find uncountable $\Gamma_1 \subset \Gamma_0$ such that 
$$\rho(\alpha_i, \alpha_j) \geq \min\{\rho(\alpha_i, \gamma), \rho(\alpha_j, \gamma) \}$$
whenever $i< j$ in $\Gamma_1$. In order to see $\Acal_l = \{ p_i: i \in \Gamma_1 \}$ works, assume $i<j$ are in $\Gamma_1$.  Then 
$$p_i(\alpha_i) \wedge p_j(\alpha_j)< t_i< \rho(\alpha_i, \gamma)=\min\{\rho(\alpha_i, \gamma), \rho(\alpha_j, \gamma) \} \leq \rho(\alpha_i, \alpha_j),$$
as desired. This finishes our induction for the singleton sets $I_l$. 

Before we deal with the the induction steps in which $I_l$ is a pair, let's make an observation.
\begin{obs}\label{dichotomy}
Assume $\mathbb{V}$ is a forcing extension of $\textsc{V}[G_{\mu \cup w}]$
by a forcing described in Lemma \ref{AddAntichain}.
Let $m<n<k$ and assume in $\mathbb{V}$,  
$\Bcal \subset  \Acal $ is uncountable 
such that 
the maps $p \mapsto p(d_p(n))$ and $p \mapsto p (d_p(m))$ are countable-to-one on $\Bcal$.
Then either 
\begin{enumerate}
\item[(a)]
there are incomparable $ s,t$  in $ T$  and uncountable $\Bcal_0 \subset \Bcal$  such that for all $p \in \Bcal_0$, $s <_T p (d_p(m))$ and $t <_T p (d_p(n))$, or 
\item[(b)] 
$\{ p( d_p(m)) \wedge p( d_p(n)): p \in \Bcal\}$ is uncountable.
\end{enumerate}
\end{obs}
\begin{proof}[Proof of observation \ref{dichotomy}]
Assume  $\{ p( d_p(m)) \wedge p( d_p(n)): p \in \Bcal\}$ is countable.
Assume  $u \in T$ such that for 
uncountably many $p \in \Bcal$, $p( d_p(m)) \wedge p( d_p(n)) =u$
and let $\delta = \Ht(u) +1$.
Then there are incomparable $s,t$ above $u$ in $T_\delta$ such that 
$$\Bcal' = \{p \in \Bcal: (p ( d_p(m))\rest (\delta+1), p ( d_p(n))\rest (\delta+1)) = (s,t)\}$$
is uncountable.
Since both maps $p \mapsto p (d_p(m))$ and $p \mapsto p ( d_p(n))$ are countable-to-one, 
there is an uncountable $\Bcal_0 \subset \Bcal'$ as desired in (a).
Therefore, the dichotomy in Observation \ref{dichotomy} holds.
\end{proof}

Assume $\textsc{V}_{l-1}, \Acal_{l-1}, I_l$ are given and $I_l = \{m,n\}$ is a pair.
Based on observation \ref{dichotomy}, 
we can assume at least one of the following cases holds:
\begin{enumerate}
\item[(0)] At least one of the maps $p \mapsto p(d_p(n))$ or 
$p \mapsto p(d_p(m))$ is not countable-to-one on $\Acal_{l-1}$.
\item[(a)] 
There are incomparable $ s,t$  in $ T$  and uncountable $\Bcal_0 \subset \Acal_{l-1}$  such that for all $p \in \Bcal_0$, $s <_T p ( d_p(m))$ and $t <_T p ( d_p (n))$.
Moreover, the maps $p \mapsto p(d_p(n))$ and $p \mapsto p(d_p(m))$
are countable-to-one on $\Acal_{l-1}$.
\item[(b.1)] The downward closure of $\{ p (d_p(m)) \wedge p( d_p(n)): p \in \Acal_{l-1}\}$  in $T$ has an uncountable branch and the maps $p \mapsto p(d_p(n))$ and $p \mapsto p(d_p(m))$
are countable-to-one on $\Acal_{l-1}$.
\item[(b.2)] The downward closure of $\{ p (d_p(m)) \wedge p( d_p(n)): p \in \Acal_{l-1}\}$ in $T$ is an Aronszajn tree and the maps $p \mapsto p(d_p(n))$ and $p \mapsto p(d_p(m))$
are countable-to-one on $\Acal_{l-1}$.

\end{enumerate}

For case (0), the forcing extension is the trivial forcing extension. Find uncountable $\mathcal{B} \subset \Acal_{l-1}$ and $t \in T$ such that one of the maps   $p \mapsto p(d_p(n))$ or $p \mapsto p(d_p(m))$
is constantly $t$ on $\Bcal$.
Let $\nu = t+1$ and let $\Acal_l \subset \Bcal$ be uncountable such that for $p \neq q$
in $\Acal_l$, $\rho \{d_p(n), d_q(m) \} > \nu$. So
for all distinct $p,q$ in $\Acal_l$,
$p (d_p(n)) \wedge q( d_q(m)) < \nu < \rho\{d_p(n),d_q(m) \}$.
By the symmetry and since we have already dealt with the one element subsets of $k$, this finishes case (0).

For case (a), the forcing extension is the trivial forcing extension. Fix $s,t, \Bcal_0$ as in (a) of Observation \ref{dichotomy}.
Let $\Acal_{l} \subset \Bcal_0$ be uncountable such that for $p \neq q$
in $\Acal_l$, $t < \rho\{d_p(n), d_q(m) \}$.
Then for all $p \neq q$ in $\Acal_l$, 
$p(d_p(n)) \wedge p (d_q(m)) = s \wedge t <t< \rho\{d_p(n), d_q(m) \}$. Because of symmetry and the fact that
we dealt with the one element sets in the previous steps,
this finishes case (a).

For case (b.1), 
the forcing extension is the trivial forcing extension.
Assume  $W$ is the downward closure of the uncountable 
set 
$\{ p(d_p(m)) \wedge p (d_p(n)): p \in \Acal_{l-1}\}$ in $T$. 
Using Lemmas \ref{AddAntichain} and \ref{no_extera_branches}, let
$\xi \in \mu \cup w$ such  that $b_\xi \subset W$.
We can find  $\{p_i : i \in \omega_1\} \subset \Acal_{l-1}$ 
such that 
$\Seq{p_i (d_{p_i}(m)) \wedge p_i( d_{p_i}(n)) \wedge b_\xi : i \in \omega_1}$ is strictly increasing.
Find uncountable $\Gamma_0 \subset \omega_1 $ such that the sequences \begin{itemize}
\item $\Seq{\alpha_i =d_{p_i}(n) : i \in \Gamma_0}$, 
\item 
$\Seq{\beta_i =d_{p_i}(m) : i \in \Gamma_0}$, 
\item
$\Seq{\{(p_i(\alpha_i) \wedge b_\xi), (p_i(\beta_i) \wedge b_\xi) \}: i \in \Gamma_0}$, 
\item $\Seq{\{ \rho \{\alpha_i, \xi \}, \rho \{\beta_i, \xi \}\}: i \in \Gamma_0}$
\end{itemize}
are all strictly increasing.
Find uncountable $\Gamma_1 \subset \Gamma_0$ such that 
\begin{equation} \label{root1}
\rho \{\alpha_i, \beta_j \} \geq \min \{\rho\{ \alpha_i, \xi \}, \rho \{\beta_j, \xi \} \},
\end{equation}

for $i \neq j$ in $\Gamma_1$. 

Assume $i < j$ are in $\Gamma_1$. Then 
$$p_i(\alpha_i) \wedge p_j(\beta_j)= p_i(\alpha_i) \wedge b_\xi < \rho\{\alpha_i, \xi \} =\min \{ \rho \{\alpha_i, \xi\},  \rho \{ \beta_j, \xi \} \}.$$
From (\ref{root1}) it follows that $p_i(\alpha_i) \wedge p_j(\beta_j) < \rho\{ \alpha_i, \beta_j \}$. 
Again, by symmetry and the fact that we have already dealt with the one element sets 
before, 
$\Acal_l = \{ p_i : i \in \Gamma_1 \}$ and $\textsc{V}_l = \textsc{V}_{l-1}$ works.
This finishes case (b.1).

In case (b.2), let $W$ be the downward closure of the uncountable set 
$\{ p(d_p(m)) \wedge p ( d_p(n)): p \in \Acal_{l-1}\}$ in $T$. 
Let $W'$ be an uncountable downward closed pruned subtree of $W$.
Let $\textsc{V}_l$ 
be a forcing extension of $\textsc{V}_{l-1}$ in which 
$W'$ has an uncountable antichain $A$ and which 
preserves $\omega_2$ and all the Aronszajn trees of 
$\textsc{V}_{l-1}$. 
Work in $\textsc{V}_l$ and 
let $\{ t_i : i \in \omega_1 \} \subset A$ such that 
$\Seq{\Ht_T (t_i) : i \in \omega_1}$ is strictly increasing.
Let $\gamma \in \omega_2$ be above all ordinals in
$ \{p(d_p(n)) + p( d_p(m)) : p \in \Acal_{l-1} \}$.
For each $i \in \omega_1$, find $p_i \in \Acal_{l-1}$
such that 
\begin{itemize}
\item $t_i <_T (p_i(\alpha_i) \wedge p_i(\beta_i)) $ where $\alpha_i = d_{p_i}(n)$ and $\beta_i = d_{p_i}(m)$,
\item $t_i \in \rho(\alpha_i, \gamma)$, and 
\item $t_i \in \rho (\beta_i, \gamma)$.
\end{itemize}
This is possible because the maps $p \mapsto p(d_p(n))$ and $p \mapsto p(d_p(m))$ are countable-to-one and 
$W'$ is pruned. 
Let $\Gamma_0 \subset \omega_1$ be uncountable such that:
\begin{itemize}
\item $\rho \{ \alpha_i, \beta_j  \} \geq \min \{ \rho(\alpha_i, \gamma)     , \rho(\beta_j , \gamma)  \}$
for all distinct $i,j$ in $\Gamma_0$, and
\item $\Seq{\{\rho(\alpha_i, \gamma)    ,\rho(\beta_i, \gamma) \}: i \in \Gamma_0}$ is strictly increasing.
\end{itemize}
Now we show that $\Acal_l = \{ p_i: i \in \Gamma_0 \}$ works.
Assume $i<j$ in $\Gamma_0$. Then 
$$p_i(\alpha_i) \wedge p_j(\beta_j) < t_i \in \rho(\alpha_i, \gamma) = \min \{ \rho(\alpha_i, \gamma), \rho(\beta_j, \gamma) \} \leq \rho(\alpha_i, \beta_j).$$
As in the previous case, by symmetry and the fact that we have already dealt  with the one element $I_l$'s, this finishes the work for case (b.2).

Work in  $\textsc{V}_l$. Let $\Acal'= \Acal_l$, and $b$ be the set 
of all $ u \in U$ such that $p \rest (\dom(p) \setminus w) \in \Acal'$ for some $p \in E_u$.
Obviously, $b \in \textsc{V}_l$.
The uncountability of $b$ follows from the fact that $\Acal'$ is uncountable and $E_u$ is countable for every $u \in U$.

In order to see $b$ consists of pairwise compatible 
elements of $U$,  assume $u$ and $v$ are in $b$.
Let $p,q$ in $ \Gamma$ witness that $u$ and $v$ are in $b$.
Then by definition of $\Acal$, $p \rest w$ and $q \rest w$
are in $G_{\mu \cup w}$.
In particular, $p \rest w$ is compatible with $q \rest w$.
On the other hand $p \rest (\dom(p) \setminus w)$ is compatible 
with $q \rest (\dom(q) \setminus w)$, since they are in $\Acal'$.
By the third requirement on $E_u$ and $E_v$, the conditions 
$p,q$ are compatible. But then their common extension, 
which is also an extension of $r$, forces that $u,v$ are 
comparable in $U$. 
This is because $\tau$ is forced by $r$ to become a branch.
This shows that the downward closure 
of $b$ is a cofinal branch of $U$ in $\textsc{V}_l$.
By Lemma \ref{AddAntichain}, $\textsc{V}_l$ is an Aronszajn tree 
preserving forcing extension of $\textsc{V}[G_{\mu \cup w}]$. 
So $U$ cannot be Aronszajn in $\textsc{V}[G_{\mu \cup w}]$.
This contradicts (\ref{Xi}).
\end{proof}

Now we are ready to prove Theorem \ref{main}.
Assume $\lambda$ is the first inaccessible cardinal in 
$\textsc{L}$ and $\textsc{V}$ is the generic extension of $\textsc{L}$ by 
the Levy collapse forcing with countable conditions which makes $\lambda$ the second uncountable 
cardinal. 
Note that $\textsc{V}$ is a model of $\square_{\omega_1}$.
Assume $G \subset Q$ is $\textsc{V}$-generic
and $T, \Seq{b_\xi: \xi \in \lambda}$ are the generic tree 
and branches that are defined from $G$ as usual.
We show for every Kurepa tree $K$ in $\textsc{V}[G]$
there is a Kurepa subtree of $T$
which club embeds into $K$.
By Theorem \ref{AinK}, this finishes the proof of Theorem \ref{main}.
 
Assume for a contradiction that $K \in \textsc{V}[G]$ is a Kurepa tree,  
$\dot{K}$ is a $Q$-name for $K$,
and $p_0 \in G$ forces that 
$\dot{K}$ is a Kurepa tree such that no Kurepa subtree of $\dot{T}$
club embeds into $\dot{K}$.
Let 
$\mu_0 \in \omega_2$ such that $Q_{\mu_0} \lhd Q$, $K $ and $T$ 
are in $ V[G_{\mu_0}]$
and $p_0 \in G_{\mu_0}$.
Note that in $\textsc{V}[G_{\mu_0}]$, 
\begin{equation}\label{Ass}
R_{\mu_0, \omega_2} \Vdash
``\textrm{no Kurepa subtree of } \check{T} \textrm{ club embeds into } \check{K}."
\end{equation}

Let $Y \in \textsc{V}[G_{\mu_0}]$ be the set of all $(\tau, p, x, A)$ such that:
\begin{itemize}
\item[$(a_0)$] $x$ is a finite subset of $[\mu_0, \omega_2)$, 
\item[$(a_1)$] $\tau$ is an $R_{\mu_0}(x)$-name,
\item[$(a_2)$] $p \Vdash_{R_{\mu_0}(x)}``\tau$ is a cofinal 
branch of $\check{K}$ which is not in 
$\textsc{V}[G_{\mu_0} * \dot{H}_{x'} ]$, for any finite $x'$ which is a proper subset of $x$", 
where $\dot{H}_{x'}$ is the canonical name for the $\textsc{V}[G_{\mu_0}]$-generic filter of
$R_{\mu_0}(x')$,
\item[$(a_3)$] $p$ is a one-to-one function, $\dom(p)=x$
and $\range(p)$ consists 
of the elements of the same height in $T$,
\item[$(a_4)$] $A= \{ u \in K: \exists q \in R_{\mu_0}(x)$ $q\leq p \wedge q \Vdash \check{u} \in \tau \}$.
\end{itemize}
For $i \in \{1,2,3,4\}$ let $Y_i$ be the projection of $Y$ on the $i$'th
component. 
By Lemmas \ref{small_extension} and \ref{finite}, $|Y_3| = \aleph_2$. 
Let $\Seq{x_\xi: \xi \in \omega_2}$ be an enumeration of $Y_3$.


Let $n \in \omega$ and
$\Gamma_0 \subset \omega_2$ be of size $\aleph_2$
such that $\{x_\xi: \xi \in \Gamma_0 \}$ is a $\Delta$-system 
with root $w$ and 
$|x_\xi|= n+|w|$
for  $\xi \in \Gamma_0$.
By thinning $\Gamma_0$ out if necessary 
we can assume that $\Seq{y_\xi = x_\xi \setminus w : \xi \in \Gamma_0}$ is strictly increasing.
For every $\xi \in \Gamma_0$ let $\tau'_\xi, p'_\xi, A'_\xi$ be such that
$(\tau'_\xi, p'_\xi, x_\xi, A'_\xi) \in Y$.
By thinning $\Gamma_0$ out again we assume that 
for all $i \in n + |w|$, 
$p'_\xi(x_\xi(i))$ does not depend on $\xi$.
There is a condition $r \in R_{\mu_0, \omega_2}$ which forces that 
for $\aleph_2$ many $\xi \in \Gamma_0$, $p'_\xi$ is in the generic filter $\dot{H}_{[\mu_0, \omega_2)}$, 
where $\dot{H}_{[\mu_0, \omega_2)}$ is the canonical name 
for the $\textsc{V}[G_{\mu_0}]$-generic filter of $R_{\mu_0, \omega_2}$. 
In order to contradict (\ref{Ass}), we need to work with a 
$\textsc{V}[G_{\mu_0}]$-generic filter of $R_{\mu_0, \omega_2}$
which intersects $\{ p'_\xi : \xi \in \Gamma_0\}$ on a set of size $\aleph_2$.
Due to similarity of arguments and for easier notation let's assume 
without loss of generality that 
\begin{equation}\label{large}
|G \cap \{ p'_\xi : \xi \in \Gamma_0\}| = \aleph_2.
\end{equation}

Fix $\mu \in \omega_2 \setminus \mu_0$  above $\max(w)$ such that $Q_\mu \lhd Q$.
From now on, we work in $\textsc{V}[G_\mu]$ unless otherwise stated.
Define  $\Gamma_1 \in \textsc{V}[G_\mu]$ to be the set of all $\xi \in \Gamma_0$ such that 
$\min(y_\xi) > \mu$ and 
$p'_\xi \rest w \in G_\mu$.
Obviously  $|\Gamma_1| = \aleph_2$ by (\ref{large}).
For each $\xi \in \Gamma_1$ let 
$p_\xi = p'_\xi \rest y_\xi$.
Note that by $(a_3)$ and the definition of $\Gamma_1$, $p_\xi$ is compatible 
with every condition in $G_\mu$.
Via the natural transition of objects $\tau'_\xi, A'_\xi$  
from $\textsc{V}[G_{\mu_0}]$ to 
$\textsc{V}[G_\mu]$,
we can find 
$\tau_\xi, A_\xi$ in $\textsc{V}[G_\mu]$ such that 
for all $\xi \in \Gamma_1$
the statement $(a_i)$ above 
implies $(b_i)$ below:
\begin{itemize}
\item[$(b_1)$] $\tau_\xi$ is an $R_{\mu}(y_\xi)$-name,
\item[$(b_2)$] $p_\xi \in R_{\mu}(y_\xi)$ forces that $\tau_\xi$ is a cofinal 
branch of $\check{K}$ which is not in $\textsc{V}[G_{\mu}]$, 
\item[$(b_3)$] $p_\xi$ is a one-to-one function and the elements in $\range(p_\xi)$ have the same height in $T$,
\item[$(b_4)$] $A_\xi= \{ u \in K: \exists q \in R_{\mu}(y_\xi)$ $q\leq p_\xi \wedge q \Vdash \check{u} \in \tau_\xi \}$.
\end{itemize}
We only show how we obtain $ (b_2)$.
Assume for a contradiction that $\xi \in \Gamma_1$,
$r \in G_\mu \cap R_{\mu_0, \mu}$ is an extension of  $p'_\xi \rest w$
and $\bar{p}_\xi \in R_{\mu, \omega_2}$ is an extension of $p_\xi$ such that:
$$r * \bar{p}_\xi\Vdash_{R_{\mu_0, \omega_2}} \tau'_\xi \textrm{ is a cofinal branch in } 
\textsc{V}[G_{\mu_0}*\dot{H}_{\mu_0, \mu}].$$
Since $r*\bar{p}_\xi$ extends $p'_\xi$, by $(a_2)$, 
$$r * \bar{p}_\xi\Vdash_{R_{\mu_0, \omega_2}} \tau'_\xi \textrm{ is a cofinal branch in } \textsc{V}[G_{\mu_0}*\dot{H}_{\mu_0, \mu}] \cap \textsc{V}[G_{\mu_0} * \dot{H}_{x_\xi}].$$
This contradicts $(a_2)$ because 
by Lemma \ref{no_extera_branches},
for every $\textsc{V}[G_{\mu_0}]$-generic filter $H \subset R_{\mu, \omega_2}$, 
$\textsc{V}[G_{\mu_0}* H_x] \cap \textsc{V}[G_{\mu_0} * H_{\mu_0, \mu}] = \textsc{V}[G_{\mu_0} * H_{x \cap \mu}]$ and 
$x_\xi \cap \mu$ is a proper subset of $x_\xi$.
Hence $(b_2)$ holds.

Note that by Lemma \ref{small_extension}, all finite powers of $T$ and $K$ have at most $\aleph_1$ many cofinal branches and 
Souslin subtrees in $\textsc{V}[G_\mu]$.
Let $\Gamma_2 \subset \Gamma_1$ be of size $\aleph_2$ such that for all $\xi $ and $\eta$ in $\Gamma_2$ the following hold:
\begin{itemize}
\item $S^\mu[y_\xi(i)] = S^\mu[y_\eta(i)]$ for all $i \in n $, 
\item $S^\mu[y_\xi] = S^\mu [y_\eta]$,
\item $A_\xi = A_\eta$.
\end{itemize}
Observe that if $y \in \{ y_\xi: \xi \in \Gamma_2 \}$ and 
$\bar{v}=\Seq{v_i: i \in n}$ is an element of $ S^\mu[y] $,
and $v_i$'s are pairwise distinct
then
$\bigotimes\limits_{i \in n}(S^\mu[y(i)])_{v_i} = (S^\mu[y])_{\bar{v}}$.
Moreover, this tree does not depend on the choice of $y$.
For $i \in n$, let $t_i = p_\xi (y_\xi(i))$
for some $\xi \in \Gamma_2$. 
The properties of  $\Gamma_0$ guarantees that $t_i$ does not 
depend on the choice of $\xi$.
Let $\Gamma_3 \subset \Gamma_2$ with $|\Gamma_3| = \aleph_2$
such that
 if $\xi < \eta$ are in $\Gamma_3$, $\alpha \in y_\xi,$ 
$\beta \in y_\eta$, then $\rho(\alpha, \beta) > \max\{t_i : i \in n \}$.

For every $\zeta \in \Gamma_3$ define $\varphi_\zeta$ from    
$\bigotimes\limits_{i \in n}(S^\mu[y(i)])_{t_i}$
to the poset consisting of all extensions of 
$p_\zeta= \{ (y_\zeta(i), t_i): i \in n \}$ 
in $R_{\mu}( y_\zeta)$
as follows. 
For every $\bar{s}= \Seq{s_i: i \in n}$ in $ \bigotimes\limits_{i \in n}(S^\mu[y(i)])_{t_i}$, 
let $\varphi_\zeta(\bar{s})$ be the function defined on $y_\zeta$
which sends $y_\zeta(i)$ to $s_i$.
It is easy to see that $\varphi_\zeta$ is  an isomorphism from its domain to a dense subset of 
the set of all extensions of $p_\zeta$ in $R_\mu( y_\zeta)$.
Let $S = \bigcup\limits_{i \in n}(S^\mu[y(i)])_{t_i}$ 
and $U=A_\zeta$.
Obviously, $U$ is Souslin in $\textsc{V}[G_\mu]$.
Also $\textsc{V}[G_\mu]$ thinks that there is a derived tree of $S$, namely $ \bigotimes\limits_{i \in n}(S^\mu[y(i)])_{t_i}$, which adds a branch to $U$.
\begin{claim}\label{free}
All derived trees of $S$ are Souslin in $\textsc{V}[G_\mu]$.
\end{claim}
\begin{proof}
Assume $\Seq{s^i_j:i \in n \wedge j \in m}$
are pairwise distinct elements of $S$ with the same height $\delta$ such that 
$t_i \leq s^i_j $
for all $i,j$.
We will show that $\prod\{S_{s^i_j}: i \in n \wedge j \in m \}$
is the set of all possible points of a branch  of $T^{[mn]}$
which is added by a ccc poset in $\textsc{V}[G_\mu]$. 
Let $\Seq{\xi_j: j \in m}$ be a strictly increasing sequence  
in $\Gamma_3$
such that for all $j<k<m$
if $\alpha \in y_{\xi_j}$ and $\beta \in y_{\xi_k}$ then $\rho(\alpha, \beta) > \delta + \omega$.
Let $z_j = y_{\xi_j}$.
Define $p : \bigcup\limits_{j \in m}z_j  \longrightarrow T$ by 
$p(z_j(i))= s^i_j$. 
By the requirement on $\Gamma_3$ and the fact that $\varphi_{\xi_j}$ is an isomorphism, 
$p \rest z_j \in R_\mu( z_j)$ for all $j \in m.$
The way we chose the  $z_j$'s implies 
that $p \in R_\mu( \bigcup\limits_{j \in m}z_j)$.

Obviously, the set of all extensions of $p$ in $R_\mu( \bigcup\limits_{j \in m}z_j)$
is a ccc poset in $\textsc{V}[G_\mu]$ and it adds a new branch to 
$T^{[mn]}$.
We show that the set
$\prod\{S_{s^i_j}: i \in n \wedge j \in m \}$
is the set of all possible points of this branch.
In order to see this, assume $a^i_j \geq s^i_j$ is in $S^{\mu}[y(i)]$.
Then the function $r$ on $ \bigcup\limits_{j \in m}z_j$ defined by $r(z_j(i))= a^i_j$ is a condition in $R_\mu (\bigcup\limits_{j \in m}z_j).$ 
This can be seen in the same way as we showed $p \in R_\mu(\bigcup\limits_{j \in m}z_j)$.
Moreover, $r$ forces that $\Seq{a^i_j: i \in n \wedge j \in m}$
is in the new branch that is added by $R_\mu(\bigcup\limits_{j \in m}z_j)$.
Therefore,
$\prod\{S_{s^i_j}: i \in n \wedge j \in m \}$
is the set of all possible points of the new branch that is added by 
$R_\mu(\bigcup\limits_{j \in m}z_j)$, 
which is a ccc poset in $\textsc{V}[G_\mu]$. 
This shows the derived tree of $S$ generated by
 $\Seq{s^i_j:i \in n \wedge j \in m}$
is a Souslin tree.
\end{proof}

\begin{claim}\label{everywhere}
Assume $\Seq{v_j: j \in k}$ is  a sequence of pairwise distinct elements of the same height in $S$.
Then in $\textsc{V}[G_\mu]$,  there is a condition $q$ in 
$R_{\mu,\omega_2}$ which forces that each $S_{v_j}$ is Kurepa.
\end{claim}
\begin{proof}
Fix $\Gamma_4 \subset \Gamma_3$ such that $|\Gamma_4|= \aleph_2$ and for all $\xi < \eta$ in $\Gamma_4$, for all $\alpha \in y_\xi$, for all $\beta \in y_\eta$, 
$$\rho(\alpha , \beta) > \max\{v_i: i \in k \}.$$
For every increasing  $\sigma = \Seq{\xi_l: l \in k }$ in $ \Gamma_4$, let $q_\sigma : \bigcup\limits_{l \in k}y_{\xi_l} \longrightarrow S $ 
be a function such that $q_\sigma(y_{\xi_l}(i)) = v_j$ if 
$v_j$ is the $l$'th ordinal in $\Seq{v_j: j \in k}$ that is above $t_i$ in $T$.
If there is no $l$'th ordinal in $\Seq{v_j : j \in k}$ that is above $t_i$ in $T$, 
let $q_\sigma(y_{\xi_l}(i)) = t_i$.
The same argument as in Claim \ref{free} shows that $q_\sigma \in R_{\mu, \omega_2}.$

Let $\Gamma_5 \subset [\Gamma_4]^k$ be a collection 
of pairwise disjoint sets with $|\Gamma_5| = \aleph_2$.
Since $R_{\mu, \omega_2}$ is ccc,
there is a condition $q \in R_{\mu, \omega_2}$
which forces that for $\aleph_2$ many $\sigma \in \Gamma_5$, 
$q_\sigma$ is in the generic filter.
But then $q$ forces that  $S_{v_j}$ is Kurepa for all $j \in k$.
\end{proof}

In $\textsc{V}[G_\mu],$ let $\bigotimes\limits_{i \in k} S_{v_i}$ be a derived tree of  $S$ which adds a branch to $U$ and 
which has the minimum dimension with this property.
Such a derived tree exists because 
$\bigotimes\limits_{i \in n} S_{t_i}$ adds a branch to $U$.
By Lemma \ref{AS} and Claim \ref{free}, 
there is a club embedding $f$ from 
 $\bigotimes\limits_{i \in k} S_{v_i}$ to $U$ in $\textsc{V}[G_\mu]$.
By Claim \ref{everywhere}, there is a condition $q \in R_{\mu, \omega_2}$ which forces that all $S_{v_i}$'s are Kurepa subtrees 
of $T$ in $\textsc{V}[G]$.
Let $j \in k$ and  $c_i \in \textsc{V}[G]$ be a cofinal branch of $S_{v_i}$, for $i \in k \setminus \{j \}$. 
In $\textsc{V}[G]$ let $g$ be the restriction of $f$ to the tree 
$(\bigotimes\limits_{i \in k \setminus \{j \}}c_i) \bigotimes S_{v_{j}}$, 
which obviously is isomorphic to $S_{v_{j}}$.
Then  $q$ forces that $g$ is  a club embedding from an isomorphic copy of $S_{v_{j}}$ into $U$, and $S_{v_{j}}$ is a Kurepa subtree of $T$. 
This contradicts (\ref{Ass}).

Acknowledgement. \\ 
The research on the second author is partially supported by grants from NSERC(455916), SFS (7750027) and CNRS(UMR7586).

The authors would like to thank the anonymous referee for 
careful comments and pointing out important mistakes in the 
previous version of the article.

\def\Dbar{\leavevmode\lower.6ex\hbox to 0pt{\hskip-.23ex \accent"16\hss}D}


\begin{thebibliography}{1}

\bibitem{club_isomorphic}
U.~Abraham and S.~Shelah.
\newblock Isomorphism types of {A}ronszajn trees.
\newblock {\em Israel J. Math.}, 50(1-2):75--113, 1985.

\bibitem{embedding_Atrees_rationals}
J.~Baumgartner, J.~Malitz, and W.~Reinhardt.
\newblock Embedding trees in the rationals.
\newblock {\em Proceedings of the National Academy of Sciences},
  67(4):1748--1753, 1970.

\bibitem{no_real_Aronszajn}
T.~Ishiu and J.~T. Moore.
\newblock Minimality of non $\sigma$-scattered orders.
\newblock {\em Fund. Math.}, 205(1):29--44, 2009.

\bibitem{Mahlo_gen_ccc}
R.~B. Jensen and K.~Schlechta.
\newblock Results on the generic {K}urepa hypothesis.
\newblock {\em Arch. Math. Logic}, 30(1):13--27, 1990.

\bibitem{Komjath_Aronszajn_Kurepa}
P.~Komj\'{a}th.
\newblock Morasses and the {L}\'{e}vy-collapse.
\newblock {\em J. Symbolic Logic}, 52(1):111--115, 1987.

\bibitem{first}
H.~Lamei~Ramandi and J.~Tatch~Moore.
\newblock There may be no minimal non-{$\sigma$}-scattered linear orders.
\newblock {\em Math. Res. Lett.}, 25(6):1957--1975, 2018.

\bibitem{walks}
S.~Todorcevic.
\newblock {\em Walks on ordinals and their characteristics}, volume 263 of {\em
  Progress in Mathematics}.
\newblock Birkh\"auser, 2007.

\end{thebibliography}
\end{document}